\documentclass[12pt]{amsart}
\usepackage[headings]{fullpage}
\usepackage{amssymb,amsmath,amscd,bbm,tikz}

\usetikzlibrary{decorations.markings}
\usetikzlibrary{decorations.pathreplacing}
\usetikzlibrary{arrows,shapes,positioning}
\tikzstyle directed=[postaction={decorate,decoration={markings,
    mark=at position #1 with {\arrow{>}}}}]
\tikzstyle rdirected=[postaction={decorate,decoration={markings,
    mark=at position #1 with {\arrow{<}}}}]

\tikzstyle{mid>}=[decoration={markings, mark=at position 0.5 with
{\arrow{>}}}, postaction={decorate}]
\tikzstyle{mid<}=[decoration={markings, mark=at position 0.5 with
{\arrow{<}}}, postaction={decorate}]
\tikzstyle{upper>}=[decoration={markings, mark=at position 0.8 with
{\arrow{>}}}, postaction={decorate}]
\tikzstyle{upper<}=[decoration={markings, mark=at position 0.8 with
{\arrow{<}}}, postaction={decorate}]
\tikzstyle{lower>}=[decoration={markings, mark=at position 0.2 with
{\arrow{>}}}, postaction={decorate}]
\tikzstyle{lower<}=[decoration={markings, mark=at position 0.2 with
{\arrow{<}}}, postaction={decorate}]

\usepackage{graphicx}
\usepackage[all]{xy}
\usepackage{url}
\usepackage[bookmarks=true,%
    colorlinks=true,%
    linkcolor=blue,%
    citecolor=blue,%
    filecolor=blue,%
    menucolor=blue,%
    urlcolor=blue,%
    breaklinks=true]{hyperref}

\newtheorem{theorem}{Theorem}[section]
\theoremstyle{definition}
\newtheorem{proposition}[theorem]{Proposition}
\newtheorem{lemma}[theorem]{Lemma}
\newtheorem{definition}[theorem]{Definition}
\newtheorem{remark}[theorem]{Remark}
\newtheorem{corollary}[theorem]{Corollary}

\newtheorem{question}[theorem]{Question}

\def\BN{\mathbbm N}
\def\BZ{\mathbbm Z}
\def\BQ{\mathbbm Q}
\def\BR{\mathbbm R}
\def\BC{\mathbbm C}

\def\calX{\mathcal X}
\def\calP{\mathcal P}

\def\la{\langle}
\def\ra{\rangle}

\def\al{\alpha}
\def\be{\begin{equation}}
\def\ee{\end{equation}}

\def\lprod{\operatorname{\overleftarrow\prod}}
\def\rprod{\operatorname{\overrightarrow\prod}}

\def\cA{\mathcal A}
\def\fB{\mathfrak B}
\newcommand{\qbinom}[2]{\text{$\left[\begin{array}{c}#1\\ #2\end{array}
\right]$}}

\def\fsl{\mathfrak{sl}}
\def\Hom{\mathrm{Hom}}
\def\ev{\mathrm{ev}}
\def\HOMFLY{HOMFLYPT}
\def\sL{\mathsf{L}}
\def\sM{\mathsf{M}}

\def\cR{\mathcal{R}}
\def\Ann{\mathrm{Ann}}
\def\GL{\mathrm{GL}}
\def\cW{\mathbb W}
\def\calF{\mathcal{F}}
\def\calC{\mathcal{C}}
\def\La{\Lambda}
\def\Sym{\mathrm{Sym}}
\def\calS{\mathcal{S}}

\newcommand{\nWeb}{\cat{nWeb}}

\renewcommand{\to}{\rightarrow}

\newcommand{\hackcenter}[1]{
 \xy (0,0)*{#1}; \endxy}
\newcommand{\scs}{\scriptstyle}
\newcommand{\xsum}[2]{
  \xy
  (0,.4)*{\sum};
  (0,3.7)*{\scs #2};
  (0,-2.9)*{\scs #1};
  \endxy
}
\newcommand{\refequal}[1]{\xy {\ar@{=}^{#1}
(-1,0)*{};(1,0)*{}};
\endxy}

\newcommand{\und}{\underline}
\newcommand\no[1]{}

\def\fg{{\mathfrak g}}
\def\Zxq{\BQ(q)[x^{\pm 1}]}

\def\Zq{\BZ[q^{\pm 1}]}
\def\cR{\mathcal R}
\def\Ass{\mathfrak{A}}
\def\hAss{\widehat{\Ass}}
\def\Qq{{\BQ(q)}}

\def\fgl{\mathfrak{gl}}

\def\ve{\varepsilon}
\def\Capam{\mathsf{\Lambda}}
\def\Cupam{\mathsf{V}}
\def\sY{\mathsf{Y}}

\def\cE{{\mathcal E}}

\def\Ext{\operatorname{Ext}}
\def\cR{\mathcal R}
\def\cT{\mathcal T}

\def\codim{\mathrm{codim}}
\def\EV{\mathrm{eval}}

\def\lef{{\mathrm{left}}}
\def\rig{{\mathrm{right}}}
\def\Lin{\mathrm{Lin}}

\def\nRep{n\mathbf{Rep}_{\wedge}}
\def\nWeb{n\mathbf{Web}}
\def\nLad{n\mathbf{Lad}}
\def\ntRep{n\mathbf{Rep}}

\def\Lad{\mathbf{Lad}}
\def\Ladnm{n\mathbf{Lad}_m}

\def\Uqgn{\mathrm{U}_q(\mathfrak{gl}_n)}

\def\Uqgm{\mathrm{U}_q(\mathfrak{gl}_m)}
\def\Uqgmd{\dot{\mathrm{U}}_q(\mathfrak{gl}_m)}
\def\Uqsn{\mathrm{U}_q(\mathfrak{sl}_n)}

\def\len{\mathrm{length}}

\def\Qqn{\BQ(q^{1/n})}
\def\Zqt{\BZ[q^{\pm 2}]}
\def\Qxq{\BQ(q)[x^{\pm 1}]}
\def\ot{\otimes}

\def\Eir{E_i^{(r)}}
\def\Fir{F_i^{(r)}}

\def\tJ{\tilde J}

\def\He{\mathrm{He}}

\def\eval{\mathrm{eval}}
\def\La{\Lambda}
\def\ot{\otimes}

\def\tnm{\vartheta(n,m)}
\def\id{\mathrm{id}}

\def\Lcl{\mathrm{Lcl}}
\def\cl{\mathrm{cl}}

\def\pabn{p^n_{a,b}}

\def\w{\mathrm{w}}
\def\Tail{\mathrm{Tail}}

\def\Capp{\mathrm{Cap}}
\def\Cupp{\mathrm{Cup}}

\def\Tail{\mathrm{Tail}}

\def\sH{\mathcal{H}}

\def\cP{\mathcal{P}}
\def\Rn{\nRep}

\def\Aut{\mathrm{Aut}}
\def\tr{\mathrm{tr}}
\def\br{\mathbf{r}}
\def\bg{\mathbf{g}}
\def\oT{\mathring T}

\begin{document}
\title[The colored HOMFLYPT function is $q$-holonomic]{
The colored HOMLFYPT function is $q$-holonomic}
\author{Stavros Garoufalidis}
\address{School of Mathematics \\
         Georgia Institute of Technology \\
         Atlanta, GA 30332-0160, USA \newline
         {\tt \url{http://www.math.gatech.edu/~stavros}}}
\email{stavros@math.gatech.edu}
\author{Aaron D. Lauda}
\address{Department of Mathematics \\
        University of Southern California \\
        3620 S. Vermont Ave, KAP 108 \\
        Los Angeles, CA 90089-2532, USA \newline
{\tt \url{http://www-bcf.usc.edu/~lauda/Aaron_Laudas_Page/Home.html}}}
\author{Thang T.Q. L\^e}
\address{School of Mathematics \\
         Georgia Institute of Technology \\
         Atlanta, GA 30332-0160, USA \newline
         {\tt \url{http://www.math.gatech.edu/~letu}}}
\email{letu@math.gatech.edu}

\date{February 18, 2017}


\begin{abstract}
We prove that the HOMFLYPT polynomial of a link, colored by partitions with
a fixed number of rows is a $q$-holonomic function. Specializing to the case
of knots colored by a partition with a single row, it proves the existence
of an $(a,q)$ super-polynomial of knots in 3-space, as was conjectured
by string theorists. Our proof uses
skew Howe duality that reduces the evaluation of web diagrams
and their ladders to a Poincare-Birkhoff-Witt computation of an auxiliary
quantum group of rank the number of strings of the ladder diagram. The result is a concrete and algorithmic web evaluation algorithm that is manifestly $q$-holonomic.
\end{abstract}

\maketitle

\tableofcontents


\section{Introduction}
\label{sec.intro}

\subsection{The colored Jones polynomial}
\label{sub.jones}

The best-known quantum invariant of a knot or link $L$ in 3-space is the
Jones polynomial $J_L$, which when properly normalized, is a Laurent
polynomial in a variable $q$ with integer coefficients. Jones's
discovery of this polynomial marked the birth of quantum
topology~\cite{Jones}, and shortly afterwards a plethora of quantum
invariants of knots and links were discovered by Reshetikhin-Turaev;
see~\cite{RT} and also the books~\cite{Ohtsuki,Tu2}.

Although Jones's definition of the Jones polynomial came from
the von Neumann algebras and their subfactors, a
connection of the Jones polynomial with the simplest non-abelian simple
Lie algebra, $\fsl_2$, and its representations was soon discovered.

More precisely, given a simple Lie algebra $\fg$ and an irreducible
(finite dimensional) representation $V$ (usually called a color,
in the physics literature) and a knot $K$, the theory of ribbon category
\cite{RT,Tu2} defines an invariant $J_K^\fg(V) \in \BZ[q^{\pm 1}]$.
The original construction of this invariant was a rational
function in a fractional power of $q$, and a normalization of this
invariant was shown in \cite{Le:Integral} to be an element of
$\BZ[q^{\pm 2}]$. The Reshetikhin-Turaev construction extends to
framed oriented links as well, each component of which is colored by
an irreducible representation of $\fg$.

Specializing to $\fsl_2$, and using the well-known
fact that there is one irreducible representation $h_n$ of $\fsl_2$
of dimension $n+1$ for every natural number $n$, it follows that a knot $K$
gives rise to a sequence of polynomials $J^{\fsl_2}_K(h_n) \in \BZ[q^{\pm 1}]$
for $n=0,1,2,\dots$. This sequence, although infinite, satisfies some
finiteness property which in particular implies that it is
determined by finitely many initial terms (the number of initial terms
depends on the knot though). More precisely, it was proven by two of the
authors in~\cite{GL} that for every knot $K$ there exists a recursion
\begin{equation}
\label{eq.Jn}
c_d(q^n,q) J^{\fsl_2}_K(h_{n+d}) + c_{d-1}(q^n,q) J^{\fsl_2}_K(h_{n+d-1})
+ \dots  + c_0(q^n,q) J^{\fsl_2}_K(h_{n})
=0
\end{equation}
for all $n \in \BN$, where $d \in \BN$,
$c_j(u,v) \in \BQ[u^{\pm 1}, v^{\pm 1}]$ for all $j=0,\dots,d$ and
$c_d \neq 0$. The recursion depends on the knot, and although it is
not unique, it can be chosen canonically.

Aside from the above-mentioned finiteness statement, the importance
of this minimal recursion (often called the $\hat{A}$-polynomial)
 is not a priori clear.
Keeping in mind that $\mathrm{PSL}(2,\BC)$ is the isometry group of
orientation preserving isometries of 3-dimensional hyperbolic space,
there are at least two connections between the $\hat{A}$-polynomial
and hyperbolic geometry: (a)
specializing the coefficients of the above recursion to $q=1$, is
conjectured to recover the defining polynomial for the
$\mathrm{SL}(2,\BC)$-character variety of the knot
complement, restricted to the boundary torus of the knot complement.
This co-called AJ Conjecture is one link of the colored Jones polynomial
with the geometry of $\mathrm{SL}(2,\BC)$ representations;
see~\cite{Ga:AJ,Le:AJ14}. (b) Such a recursion can be used to numerically
several terms of the asymptotics of the colored Jones polynomial at
complex roots of unity, a fascinating story that connects quantum
topology to hyperbolic geometry and number theory. For a sample of
computations, the reader may consult~\cite{Ga3,GZ}.

Returning back to recursion relations, sequences that
satisfy a recursion relation of the form~\eqref{eq.Jn} are $q$-holonomic,
a key concept introduced by Zeilberger~\cite{Z90}. $q$-holonomic functions
enjoy several closure properties. A key theorem of
Wilf-Zeilberger is that a multisum of a $q$-proper hypergeometric term
(where we sum all but one variable) is $q$-holonomic~\cite[Thm.5.1]{WZ}.
This theorem, and the fact that quantum knot invariants are multisums of
$q$-proper hypergeometric terms (coming from structure constants of
corresponding quantum groups), explains why the quantum knot invariants
are $q$-holonomic functions.

Converting the above statement into a theorem and a proof requires
additional work. To begin with, one needs to consider functions of several
variables. For instance the $\fsl_3$-colored Jones polynomial of a knot,
or the $\fsl_2$-colored Jones polynomial of a 2-component link is a
function of two discrete variables. A definition of $q$-holonomic functions
of several variables was given by Sabbah~\cite{Sabbah} using the language
of homological algebra. Sabbah used a theory of Hilbert dimension for
modules over rings generated by $q$-commuting variables, and proved
a key Bernstein inequality. A survey of Zeilberger's and Sabbah's work was
given by two of the authors in~\cite{GL:survey}, where detailed proofs and
examples of $q$-holonomic functions is discussed. A summary of the main
definitions and properties of $q$-holonomic functions is given in
Section~\ref{sec.qholo}.

\subsection{The colored \HOMFLY\ polynomial}
\label{sub.homfly}

Shortly after the discovery of the Jones polynomial, five groups
independently discovered a two-variable polynomial, the \HOMFLY \
polynomial $W$ that takes values in the ring
$\BQ(q)[x^{\pm 1}]$~\cite{HOMFLY,PT}.
Turaev \cite{Tu1} showed that the latter unifies the quantum link invariants
$J_L^{\fsl_n}(h_1,\dots,h_1)$, where $h_1=\BC^n$ is the defining
representation of $\fsl_n$, as follows: for every $n \geq 2$ and every
framed oriented link $L$ whose components are colored by $\BC^n$, we have:
$$
\tJ_L^{\fsl_n}(h_1,\dots,h_1) = W_L|_{x=q^n}.
$$
Here $\tJ_L$ is a normalized version of $J_L$, see Section \ref{sec.3cats}.

Let $\calP$ denote the set of partitions
$\lambda=(\lambda_1,\lambda_2,\dots)$ where
$\lambda_1 \geq \lambda_2 \geq \dots \geq 0$ is a decreasing sequence
of nonnegative natural numbers, all but finitely many zero. As usual,
a partition is presented by a Young diagram. Let $\calP_{n-1}$ be the set of
partitions  with at most $n-1$ rows. Irreducible representations of $\fsl_n$
are parameterized by partitions in $\calP_{n-1}$, and we will identity a
partition $\lambda \in \cP_{n-1}$ with its corresponding irreducible
$\fsl_n$-module (which has highest weight $\lambda$, see \cite{FH}).
With this identification, the partition $h_a$, which has one row and $a$
boxes, is the $a$-th symmetric power of $h_1$, and the partition $e_a$,
which has one column and $a$ boxes, is the $a$-th external power of $h_1=e_1$.

Wenzl \cite{Wenzl}, generalizing Turaev's result, showed  that the
$\fsl_n$-quantum link invariants interpolate a two-variable function in the
following sense. If $L$ is an oriented framed link with $r$ ordered
components and $\lambda_i$ are partitions with at most $\ell$ rows
for $i=1,\dots,r$, then there exists a two-variable
colored \HOMFLY\ function $W_L(\lambda_1,\dots,\lambda_r) \in
\BQ(q)[x^{\pm 1}]$ such that for all natural numbers $n$ with $n \geq \ell+1$
we have:
$$
\tJ_L^{\fsl_n}(\lambda_1,\dots,\lambda_r) =
W_L(\lambda_1,\dots,\lambda_r)|_{x=q^n}.
$$
A detailed definition of the \HOMFLY\ polynomial and its colored version
in terms of the \HOMFLY\ polynomial of cables of the link is given
in \cite{Mo1,Mo2}.

\subsection{Statement of our results}
\label{sub.results}

The set $\cP$ of all partitions has an involution defined by
$\lambda \mapsto \lambda^\dagger$ which
transposes columns and rows of a partition.
The map $\iota_\ell : \BN^\ell \to \calP_{\ell}$ given by
$$
\iota_\ell (n_1,\dots,n_\ell)
=(\lambda_1,\dots,\lambda_\ell) \in \calP_\ell, \quad \text{where } \
\lambda_i= \sum_{j=1}^{n-i+1} n_j
$$
is a bijection, and so is
$\iota_\ell^\dagger : \BN^\ell \to \calP_{\ell}^\dagger$ (where
$\calP_\ell^\dagger$ is the set of all partitions with at most $\ell$ columns)
defined by
$\iota_\ell^\dagger(n_1,\dots,n_\ell)= (\iota_\ell(n_1,\dots,n_\ell))^\dagger$.

\begin{theorem}
\label{thm.1}
Suppose $L$ is an oriented, framed link with $r$ ordered components
and $\ell$ a nonnegative integer. Then, the following functions
$$
W_L \circ (\iota_\ell)^r : \BN^{r \ell} \to \Zxq, \qquad
W_L \circ (\iota^\dagger_\ell)^r : \BN^{r \ell} \to \Zxq
$$
are $q$-holonomic.
\end{theorem}

\begin{corollary}
\label{cor.thm1}
For a framed oriented knot $K$ colored with partitions with a single
row, the sequence $W_{K}(h_a)$ for $a=0,1,2,\dots$
is $q$-holonomic.
\end{corollary}
Some special cases of the above corollary are known; see
Cherednik~\cite{Cherednik} for the case of torus knots,
Wedrich~\cite{Wedrich} for the case of 2-bridge knots, and
Kawagoe~\cite{Kawagoe} for some 2-bridge knots and some pretzel knots.

On the set of all functions from $\BN$ to $\BQ(x,q)$ define two operators
$\sL, \sM$ by
$$
(\sL f)(a)=f(a+1), \quad (\sM f)(a) = q^a f(a)\,.
$$
Then $\sL \sM = q \sM \sL$, and a recurrence for $f$ has the form $Pf=0$,
where
$$
P = \sum_{j=0}^d c_j(q,x,M) L^j, \quad c_j(q,x,M) \in \BZ[q, x,M] \,.
$$
When non-zero recurrence for $f$ exists, there are many of them, and there
is a unique  one, up to sign, such that (i) $d$ is minimal, (ii) the total
degree in $q,x,M,L$ is minimal, and (iii) all the integer coefficients of
$P$ are co-primes, see~\cite{Ga:AJ,Le:AJ14}. For a knot $K$, we
denote such a minimal recurrence for $W_K(h_a)$ by $A_K(M,L,x,q)$.

Physicists have conjectured the existence of the $4$-variable polynomial
(see for instance the works~\cite{GSV, AV}), and have further conjectured that
when we set $q=1$, the corresponding $3$-variable polynomial $A_K(M,L,x,1)$
is equal, after some universal (i.e., knot independent) change
of variables with a $3$-variable polynomial that comes out of knot
contact homology~\cite{Ng,AVNg}. In the physics literature,
$A_K(\sM,\sL,Q,1)$ is often called the $Q$-deformed $A$-polynomial of a knot
and it appears in string theory in geometry of spectral
curves, topological strings, matrix models, and M-theory dualities.
There is a lot of literature on this polynomial following
the pioneering work of Gukov, Fuji, Stosic, Su{\l}kowski and others.
For a detailed discussion, see
\cite{AV,DG:refined,Dunfield-Gukov,Fuji1,Fuji2,GL,
Gukov:lectures,Gukov-Stosic,Gukov-Vafa}.

\begin{remark}
\label{rem.aholo}
The proof of Theorem \ref{thm.1} implies the function
$\BN \times \BN^{r \ell} \to \BZ[q^{\pm 1}]$ defined by
$$
(n, \vec{m}) \in \BN \times \BN^{r \ell} \mapsto
(W_L \circ (\iota_\ell)^r)(\vec{m})|_{x=q^n}
$$
is $q$-holonomic in all $r\ell+1$ variables.
The latter was conjectured in \cite{GV}.
\end{remark}

\subsection{An example}
\label{sub.rec31}
Suppose $K$ is the right-hand trefoil, see Figure \ref{fig:trefoil},
with 0 framing. Define

\begin{align*}
a_0&= x^4 (x^2 \sM^2-1) (q^6 x^2\sM^4 -1) \\
a_1&= q^7 (q^4 x^2\sM^4-1) (q^8 x^4\sM^8 - q^4 x^4 \sM^6 +q^2 x^4 \sM^4 + x^4 \sM^4 -q^6 x^2 \sM^4 -q^2 x^2 \sM^4  -x^2 \sM^2 +1)  \\
a_2&= -q^{18} x^2 M^6  (q^4 \sM^2-1) (q^2x^2\sM^4-1).
\end{align*}

Then for all $m \ge 0$,
\be
\label{eq.tre}
 a_2 W_K(h_{m+2})  + a_1 W_K(h_{m+1}) + a_0 W_K(h_{m})=0.
\ee
\begin{remark}
In \cite{Fuji1}, a conjectural formula for the colored HOMFLYPT function
$W_K(h_m)$ is given, for the case when $K$ is the left-hand trefoil
(and other torus knots). Based on this conjectural formula, the authors
of \cite{Fuji1}, using a computer program of Zeilberger \cite{PWZ},
found a recurrence formula for $W_K(h_m)$, which is different
from \eqref{eq.tre} since another normalization was used. In
Appendix~\ref{app.d}, we will give a proof of \eqref{eq.tre}.
\end{remark}

\subsection{Plan of the proof}
\label{sub.plan}

The quantum group invariants require familiarity with category theory,
representation theory of quantum groups as well as a understanding the
accompanying graphical notation.

In Section~\ref{sec.3cats} we discuss three categories $\nRep$, $\nWeb$
and $\nLad$ which are related to representations of quantum groups
as well as to a diagrammatic description of links and their invariants.
In Section~\ref{sec.x} we discuss how to unify the $\fsl_n$
link invariants to one that is independent of $n$.
In Section~\ref{sec.qholo} we discuss the basic definitions, examples
and properties of $q$-holonomic functions.
In Section~\ref{sec.thm1} we give the proof of Theorem~\ref{thm.1}.
The proof is concrete and algorithmic, with a detailed example for the
case of the right-handed trefoil given in Section~\ref{sub.trefoil}.
We summarize the steps here, using the notation of the proof.

\begin{itemize}
\item[(a)]
We start with a braid word representative $\beta$  whose closure $\cl(\beta)$
is the link $L$. The corresponding braid has $m$ strands and a fixed number
of letters. For the trefoil, this is given in equation~\eqref{eq.E31}.
\item[(b)]
The link is now given by joining to the braid the bottom and top part of the
closure consisting of cup/cap diagrams, respectively. We replace the bottom
part by a monomial in some operators $E_i$, the braid word by a product of
Lusztig braid operators $T_i(b)^{\pm 1}$
defined in section~\ref{sub.braiding}, and the top part by a monomial in some
operators $F_j$. For the trefoil, this is given in equation~\eqref{eq.Y31}.
\item[(c)]
Each operator $T_i(b)^{\pm 1}$ is a sum (over the integers)
of operators $E_i$ and $F_j$ (see
equations~\eqref{eq.tau+}--\eqref{eq.tau-}).
\item[(d)]
The operators $E_i$ and $F_j$ satisfy the quantum group $q$-commutation
relations given in equations~\eqref{eq.com1}--\eqref{eq.com4}, and using
those we can sort the above expressions by moving all the $E$'s to the
right and all the $F$'s to the left.
\item[(e)]
The fact that the operators $E_i$ annihilate the last
bit $1_\vartheta$, corresponding to the projection onto a highest weight
determined from the link diagram, adds a product of delta functions
in our sum.
\item[(f)]
The requirement that all weights appearing in the sum are positive introduces
Heaviside functions into the sum as explained in the proof of
Proposition~\ref{prop.tech1}.
\item[(g)]
This way, we obtain a multidimensional sum over the integers,
whose summand is a product of extended $q$-binomial coefficients of linear
forms (with integer coefficients) of the summation variables, times a sign
raised to a linear form of the summation variables. These sums are
always terminating. For the trefoil, this $6$-dimensional sum
is given in equation~\eqref{eq.trefoil}.
\item[(h)]
We show in section \ref{sec:evalqh} that such multisums are q-holonomic.
\end{itemize}

Hidden in the above algorithm is the quantum skew Howe duality \cite{CKM},
which allows us to compute colored $\fsl_n$-invariants by evaluating
ladder diagrams in $2m$ strands using
an auxiliary quantum group based on the Lie algebra $\mathfrak{gl}_{2m}$.
Steps (c)-(e) are exactly a Poincare-Birkhoff-Witt computation on
$\mathfrak{gl}_{2m}$.

To avoid any confusion or misunderstanding, in an earlier
article~\cite{Ga:homfly} one of the authors reduced the $q$-holonomicity
of the colored \HOMFLY\ polynomial to the $q$-holonomicity of the evaluation
of MOY graphs, and observed that the latter would follow from the existence
of a $q$-holonomic evaluation algorithm for MOY graphs. Unfortunately,
such an algorithm based on simplifications of MOY graphs or web diagrams
has yet to be found.

\subsection{Computations and questions}
\label{odds}

With regards to computation of the $4$-variable polynomial of a knot,
there are several formulas for the \HOMFLY\
polynomial of some links in the literature colored by partitions with
one row, see for example \cite{Kawagoe,NRS,Ito1,Ito2}. These formulas are
manifestly $q$-holonomic, as follows by the fundamental theorem of WZ
theory. Using these formulas and Wilf-Zeilberger theory, one can sometimes
compute the $4$-variable knot polynomial. For sample computations for
the case of twist knots and some torus knots, see~\cite{NRS}.

The next question is inaccessible with our methods. A positive answer
would be useful in the study of LMOV (also known as BPS) invariants of
links \cite{LMOV}. First, using linearity extend the colored \HOMFLY\
function to the case when the color of each link component is a
$\BZ$-linear combination of Young diagrams.
Let $p_a=\sum_{k=0}^a (-1)^k (k, 1^{a-k})$. Note that
$(k, 1^{a-k})$ is a hook partition with one row with $k$ boxes and
one column with $a-k$ boxes.

\begin{question}
\label{que.1}
Is it true that the \HOMFLY\ polynomial of a knot colored by $p_a$  is a
$q$-holonomic function of $a$?
\end{question}


\section{Categories, links and their invariants}
\label{sec.3cats}

Throughout the paper, $n$ will denote a natural number greater than or equal
to $2$. We will denote by $\Qqn$ the field of rational functions in an
indeterminate $q^{1/n}$, and $\Qq$ its subfield generated by $q= (q^{1/n})^n$.
Also $\Zq\subset \Qq$ will denote the ring of Laurent polynomial in $q$
with integer coefficients.

In this section we will discuss three categories $\nRep$, $\nWeb$ and
$\nLad$ which are connected by functors

\be
\label{eq.3cats}
\nLad \stackrel{\Psi_n}{\to} \nWeb \stackrel{\Gamma_n}{\to} \nRep
\ee

A ring homomorphism $f:\Qqn\to \Qqn$ (thought of as a homomorphism from
the empty set to the empty set), is the multiplication by a scalar, and
we denote this scalar by $\ev(f)\in \Qqn$.

These categories are intimately related to  diagrammatic descriptions
of framed tangles and of quantum groups.

\subsection{The quantized enveloping algebras $\Uqgn$ and $\Uqsn$}
\label{sub.Uqsln}

Consider the lattice $\BZ^n$ with the standard Euclidean inner product
$\langle \cdot, \cdot \ra$, and the root vectors
$$
\al_i=(0,\dots,0, 1, -1, 0 , \dots, 0)\in \BZ^n\,,
$$
with $1$ on the $i$-th position.

The quantized enveloping algebra $\Uqgn$ is the associative algebra over
$\Qq$ generated by $L_i, i=1,\dots, n$ and $E_i, F_i, i=1,\dots,n-1$,
subject to the relations

\begin{gather*}
L^a L^b= L^{a+b}, \quad L_{0} = 1  \\
L^a  E_j= q^{a_j - a_{j+1}} E_j L^a, \quad L^a F_j
= q^{a_{j+1} - a_{j}} F_j L^a \\
E_i^{(2)} E_{i+1} - E_i E_{i+1} E_i + E_{i+1} E_i^{(2)}
=0=F_i^{(2)}  F_{i+1} - F_i F_{i+1} F_i + F_{i+1} F_i^{(2)}\\
E_i F_j - E_j F_i = \delta_{ij} \frac{K_i-K_i^{-1}}{q-q^{-1}}\\
E_i E_j= E_j E_i, \quad F_i F_j = F_j F_i \quad \text{for } \ |i-j| >1.
\end{gather*}
Here $L^a= L_1^{a_1} \dots L_n^{a_n}$ for $a=(a_1,\dots,a_n)\in \BZ^n$,
$K_i= L_i L^{-1}_{i+1}$, and
$$
\Eir = E_i^r/[r]!, \Fir = F_i^r/[r]!, \quad
\text{where} \ [r]!:= \prod_{j=1}^r \frac{q^j-q^{-j}}{q-q^{-1}}.
$$

There is a structure of a Hopf algebra on $\Uqgn$ with the co-product and
the antipode, see e.g. \cite{Jantzen,CP,Lu}.

The quantized enveloping algebra $\Uqsn$ is the subalgebra of $\Uqgn$
generated by $E_i, F_i, K_i^{\pm 1}$, $ i=1,\dots,n-1$. Then $\Uqsn$
inherits a Hopf algebra structure from that of $\Uqgn$.

{\em A weight of $\Uqgn$} (resp., $\Uqsn$) is an element $a\in \BZ^n$
(resp., an element $a\in \BZ^n$ such that $\sum _i a_i=0$).
A $\Uqsn$-module $V$ is called a {\em weight module} (or perhaps better,
a weighted module) if
$V= \bigoplus_{a} V_{[a]}$, where each $a$ is a $\Uqsn$-weight and
$$
V_{[a]} = \{ v \in V \mid K_i(v) = q^{\langle \al_i, a \ra } v\}.
$$

For a partition $\lambda=(l_1,\dots,l_\ell)$ with
$l_1 \ge l_2\ge \dots \ge l_\ell >0$ we call $\ell=\len(\lambda)$ the length of
$\lambda$ and $|\lambda|=\sum_i l_i$ the weight of $\lambda$. Denote by
$\lambda^\dagger$ the conjugate of $\lambda$, which is the partition whose
Young diagram is the transpose of that of $\lambda$. For a thorough treatment
of partitions, see ~\cite{Mc}.
Finite-dimensional irreducible weight $\Uqsn$-modules are parameterized
by partitions $\lambda\in \cP_{n-1}$, i.e.  partitions of  length
$\le n-1$, see e.g \cite{CP,Jantzen}. For every  $\lambda\in \cP_{n-1}$
denote by  $V_\lambda$ the corresponding irreducible weight $\Uqsn$-module.

\subsection{The category of $U_q(\fsl_n)$-modules and link invariants}

The category $\ntRep$ of finite-dimensional weight $U_q(\fsl_n)$-modules is
a {\em ribbon category} \cite{Tu2}, where the braiding comes from the
universal $R$-matrix. To be precise, one needs to extend the ground field
to $\Qqn$ so that the braiding and the ribbon element can be defined.

By the theory of ribbon categories, for a framed oriented link $L$ in
3-space with $r$ ordered components and
$r$ objects $V_1,\dots, V_r$  of $\ntRep$, one can define an invariant
$$
J_L^{\fsl_n}(V_1,\dots,V_r) \in \Qqn.
$$
If $\lambda_1,\dots,\lambda_r\in \cP_{n-1}$,
we use the notation
$$
J_L^{\fsl_n}(\lambda_1,\dots,\lambda_r) =
J_L^{\fsl_n}(V_{\lambda_1},\dots,V_{\lambda_r}).
$$
It is known that a properly normalized version of $J_L^{sl_n}(V_1,\dots,V_r)$
belongs to $\Zqt$, see~\cite{MW,Le:Integral}. A special case of this
integrality phenomenon is the following.
Let $\ell_{ij}$ be the linking number between the $i$-th and the
$j$-component of $L$, with $\ell_{ii}$ the framing of the $i$-th component.
Define
\be
\label{eq.tJ}
\tJ_{L}^{\fsl_n}(\lambda_1,\dots, \lambda_r)=
q^{\frac 1n \sum_{i, j} \ell_{ij} |\lambda_i| |\lambda_j|}
J_L^{\fsl_n}(\lambda_1,\dots,\lambda_r).
\end{equation}
Then we have
\be
\label{eq.integral}
\tJ_L^{\fsl_n}(\lambda_1,\dots, \lambda_r)\in \Zq.
\ee

Not only is $\tJ_L^{\fsl_n}(\lambda_1,\dots, \lambda_r)$ a Laurent
polynomial in $q$, but it also enjoys the following stability (with
respect to the rank $n$) property.

\begin{proposition}[\cite{Wenzl}]
\label{r.unified}
There exists an invariant $W_L(\lambda_1,\dots,\lambda_r)\in \Qxq$ such that
for any $n$ greater than the length of any of $\lambda_j$, we have
$$
W_L(\lambda_1,\dots,\lambda_r)|_{x= q^n} =
\tJ_L^{\fsl_n}(\lambda_1,\dots, \lambda_r).
$$
\end{proposition}

$W_L$ is usually called the colored \HOMFLY\ function. The theorem was
first proved by Wenzl using quantum group theory. For a detailed proof
using skein theory see \cite[Theorem 11.4.18]{Lucak}. The theorem also
follows from our proof of Theorem \ref{thm.1} below. For the simplest case,
when all partitions have one box, Proposition \ref{r.unified} was first
proved by Turaev \cite{Tu1}.

\begin{remark}
The integrality \eqref{eq.integral} shows that the polynomial
$P=W_L(\lambda_1,\dots,\lambda_r)\in \Qxq$ has the property that
$P|_{x = q^n}\in \Zq$ for all integer $n >1$. Such a polynomial
$P\in \Qxq$ is called {\em $q$-integral} and is studied in
\cite[Sec. 2.3]{BCL}.
\end{remark}

\begin{remark}
\label{rem.2}
Our $W_L(\lambda_1,\dots,\lambda_r)$ is equal to
$P(L * (Q_{\lambda_1}, \dots, Q_{\lambda_r}))$ in the notation of
\cite[Section 6]{Mo2}, with our $q$ and $x$ equal to  respectively $s$
and $v^{-1}$ there.
\end{remark}

\subsection{Properties of the colored \HOMFLY\ polynomial}

Let $\La$ be the free $\BQ$-vector space with basis the set $\cP$ of all Young
diagrams, including the empty one. Suppose $L$ is a framed oriented link
with $r$ ordered components. The invariant $W_L(\lambda_1,\dots,\lambda_r)$ can
be extended to a $\BQ$-multi-linear map
$$
W_L: \La^r \to \Qxq.
$$
There is a $\BQ$-algebra structure on  $\La$ which makes it isomorphic to
the algebra of symmetric functions, see e.g. \cite{Mc}. Under this
isomorphism, a Young diagram $\lambda$ is mapped to the Schur function
$S_\lambda$ corresponding to $\lambda$.

We collect here some well-known properties of the quantum invariant $W_L$.

\begin{proposition}
\label{r.WL}
Let $L$ be a framed oriented link in the 3-space with $k$ ordered
components.
\newline
\rm{(a)}
Suppose $L'$ is the same $L$ with the components renumbered by a
permutation $\sigma$ of $\{1,\dots, k\}$. Then
$$
W_{L'}(\lambda_1,\dots,\lambda_r) =
W_L(\lambda_{\sigma 1},\dots,\lambda_{\sigma r}).
$$
\rm{(b)}
Suppose $\Delta L$ is the result of replacing the first component of $L$ by
two copies of its parallel push-off (using the framing). Then
$$
W_{\Delta L}(\lambda_1', \lambda_1'', \lambda_2, \dots, \lambda_r) =
W_L(\lambda_1  \lambda_1', \lambda_2, \dots, \lambda_r).
$$
\rm{(c)}
We have:
\be
W_{L}(\lambda_1, \dots, \lambda_r)=
W_{L}(\lambda_1^\dagger, \dots, \lambda_r^\dagger)|_{q\to -q^{-1}}.
\ee
\end{proposition}

Parts (a) and (b) follow from the corresponding properties for $J_L$,
see~\cite{Tu2}. While (a) is trivial, (b) follows from the hexagon equation
of the braiding in the braided category.
Part (c) is well-known and has been discussed in many papers, see
e.g.~\cite[Equ. 4.41]{LMOV}. For completeness, we give proofs of parts (b)
and (c) in Appendix \ref{app.c}.

\subsection{The category $\nRep$}
\label{sub.nRep}

Let $e_a$ be the partition whose Young diagram is a column with $a$
boxes, i.e. $e_a= (1^a)$ in the standard notation of partitions.
The $\Uqsn$-module $V_{e_a}$ with $ 1 \le a \le n-1$ is called a
{\em fundamental $\Uqsn$-module}. We also use $e_0$ to denote the
empty Young diagram, which corresponds to the trivial $\Uqsn$-module.

Let $\nRep$ be the full subcategory of $\ntRep$  whose objects are those
isomorphic to tensor products of the fundamental $\Uqsn$-modules.
Then $\nRep$ inherits a ribbon category structure from $\ntRep$.

The advantage of $\nRep$ is that it has a remarkable presentation using
planar diagrams called spider webs described in the next section. Since
$\ntRep$ is the idempotent completion of $\Rn$, we don't loose much
working with $\Rn$.

\subsection{The category $\nWeb$}
\label{sub.nwebs}

We describe here the category $\nWeb$ of $\fsl_n$-webs, following
Cautis-Kamnitzer-Morrison~\cite{CKM}.
Recall that a pivotal monoidal category is a category with tensor products
and a coherent notion of duality in which the double dual functor is
naturally isomorphic to the identity. The morphisms and the relations
among morphisms of such categories afford a diagrammatic description using
planar diagrammatics. They are essentially equivalent to the
description of the Temperley-Lieb algebra for $n=2$ and to Kuperberg's
spider webs~\cite{Kuperberg} (for $n=3$) and the planar algebras of
Jones~\cite{Jones2}. They are also closely related to the MOY graphs of
Murakami, Ohtsuki, and Yamada~\cite{MOY}. Standard references for pivotal
categories include the books \cite[Chpt.XI]{Tu2}, \cite{Kassel} as
well~\cite[Chpt.4.7]{Etingof-tensor}.

An {\em $n$-web} is a compact subset $Z$ of the horizontal strip
$\BR\times [0,1]$ with additional data satisfying (i)-(iii).
\begin{itemize}
\item[(i)]
Each connected component of $X$ is  either an oriented circle or a directed
graph  (i.e. a finite 1-dimensional CW-complex) where the degree of each
vertex is 1, 2, or 3. Every circle component and every edge is labeled by
an integer in $[1,n-1]$.
\item[(ii)]
The set $\partial Z$ of univalent vertices of $Z$ is in the union of the
top and bottom lines of the strip and $Z\setminus \partial Z$ is in the
interior of the strip.
\item[(iii)]
Up to isotopy there are 2 types of trivalent vertices and 2 types of
bivalent vertices as in the following figure (with labeling of edges
attached to the vertex):
\begin{equation}
\label{webgen}
\xy
(0,0)*{ \begin{tikzpicture} [scale=.75, decoration={markings,
                        mark=at position 0.6 with {\arrow{>}};    }]
\draw[very thick, postaction={decorate}] (0,0) -- (0,1);
\draw[very thick, postaction={decorate}] (.875,-.5) -- (0,0);
\node at (1.2,-.5) {\small$b$};
\node at (0,1.3) {\small$a+b$};
\draw[very thick, postaction={decorate}] (-.875,-.5) -- (0,0);
\node at (-1.2,-.5) {\small$a$};
\end{tikzpicture}};
\endxy
\quad , \quad
\xy
(0,0)*{
\begin{tikzpicture} [scale=.75, decoration={markings,
                        mark=at position 0.6 with {\arrow{>}};    }]
\draw[very thick, postaction={decorate}] (0,1) -- (0,0);
\node at (0,1.3) {\small$a+b$};
\draw[very thick, postaction={decorate}] (0,0) -- (.875,-.5);
\node at (-1.2,-.5) {\small$a$};
\draw[very thick, postaction={decorate}] (0,0) -- (-.875,-.5);
\node at (1.2,-.5) {\small$b$};
\end{tikzpicture}};
\endxy
\quad , \quad
\xy
(0,0)*{
\begin{tikzpicture} [scale=.75, decoration={markings,
                        mark=at position 0.6 with {\arrow{>}};    }]
\draw[very thick, postaction={decorate}] (0,-.5) -- (0,.25);
\node at (0,-.8) {\small$a$};
\draw[very thick, postaction={decorate}] (0,1) -- (0,.25);
\node at (0,1.3) {\small$n-a$};
\draw[very thick] (0,.25) -- (.2,.25);
\end{tikzpicture}};
\endxy
\quad , \quad
\xy
(0,0)*{
\begin{tikzpicture} [scale=.75, decoration={markings,
                        mark=at position 0.6 with {\arrow{>}};    }]
\draw[very thick, postaction={decorate}] (0,.25) -- (0,-.5);
\node at (0,-.8) {\small$a$};
\draw[very thick, postaction={decorate}] (0,.25) -- (0,1);
\node at (0,1.3) {\small$n-a$};
\draw[very thick] (0,.25) -- (.2,.25);
\end{tikzpicture}};
\endxy
\end{equation}
The third and the fourth graphs depict bivalent vertices but not
trivalent vertices, as the small tag there is not officially an edge.
The tag  provides a distinguished side and makes the bivalent vertices
not rotationally symmetric.
\end{itemize}

We will declare isotopic webs to be equal.

Let $\partial_-Z=(i_1^{\ve_1}, \dots i_k^{\ve_k})$, where $i_1,\dots,i_k$ are
the labels of the edges ending on the bottom line listed from left to right,
and $\ve_j=+$ if the orientation at the $j$-th ending point is upwards and
$\ve_j=-$ otherwise. One defines $\partial_+ Z$  exactly the same way, using
the top line instead of the bottom line.

The category $\nWeb$ is the pivotal monoidal $\Qqn$-linear category
whose objects are sequences in the symbols $\{1^{\pm},\ldots,(n-1)^{\pm}\}$.
Given objects $a, b$ of $\nWeb$, the set of morphisms $\Hom_{\nWeb}(a,b)$
is the set of $\Qqn$-linear combinations of
$n$-webs $Z$ such that $\partial_-Z=a, \partial_+ Z= b$, subject to
certain local relations described in
\cite[Section 2.2]{CKM}. In \cite{CKM}, our $\nWeb$ is
denoted by ${\mathcal Sp} (\mathrm{SL}_n)$.
The tensor product $Z_1 \ot Z_2$ is obtained by placing $Z_2$ to the
right of $Z_1$. The composition $Z_1Z_2$ is the result of placing $Z_1$
atop $Z_2$, after an isotopy to make the top ends of $Z_2$ match the
bottom ends of $Z_1$.

For example, the first diagram in  \eqref{webgen} represents a morphism
from $a^+ \ot b^+=(a^+, b^+) \to (a+b)^+$, and the second one represents
a morphism from $a^- \ot b^- \to (a+b)^-$.

The monoidal unit $\nWeb$ is the empty sequence.  The planar isotopy
condition implies that the object $a^+$ is dual to the object $a^-$.
The cap and cup morphisms

\begin{equation}
\label{eq.capcup}
\begin{tikzpicture} [scale=.75, decoration={markings,
                        mark=at position 0.6 with {\arrow{>}};    }]
\draw[very thick, postaction={decorate}]
(0,.25) to[out = 0, in =90](.875,-.5);
\draw[very thick, postaction={decorate}]
(-.875,-.5) to[out=90, in =180] (0,.25);
\node at (0,.65) {\small$a$};
\end{tikzpicture}
\qquad  \qquad \quad
\hackcenter{ \begin{tikzpicture} [ xscale=.75, yscale=-.75,
decoration={markings, mark=at position 0.6 with {\arrow{<}};    }]
\draw[very thick, postaction={decorate}]
(.875,-.5) to[out = 90, in =0] (0,.25);
\draw[very thick, postaction={decorate}]
(0,.25)to[out=180, in =90] (-.875,-.5);
\node at (0,.65) {\small$a$};
\end{tikzpicture}}
\end{equation}
give rise to maps $a^+ \otimes a^- \to \emptyset$ and
$\emptyset \to a^- \otimes a^+$ that realize this duality.

For simplicity we allow diagrams to carry labels of $0$ and $n$ with
the understanding  that $n$-labelled edges connected to a trivalent vertex
should be deleted and replaced by a tag as in the cap and cup diagrams:

\begin{equation}
\label{eq_tag_cup}
\hackcenter{ \begin{tikzpicture} [scale=.75, decoration={markings,
                        mark=at position 0.6 with {\arrow{>}};    }]
\draw[very thick, postaction={decorate}] (0,0) -- (0,1);
\draw[very thick, postaction={decorate}] (.875,-.5) -- (0,0);
\node at (1.1,-.7) {\small$n-a$};
\node at (0,1.3) {\small$n$};
\draw[very thick, postaction={decorate}] (-.875,-.5) -- (0,0);
\node at (-1.2,-.7) {\small$a$};
\end{tikzpicture}}
\;\; = \;\;
\hackcenter{ \begin{tikzpicture} [scale=.75, decoration={markings,
                        mark=at position 0.6 with {\arrow{>}};    }]
\draw[very thick] (0,.25) -- (0,.5);
\draw[very thick, postaction={decorate}]
(.875,-.5) to[out = 90, in =0] (0,.25);
\node at (1.1,-.7) {\small$n-a$};
\draw[very thick, postaction={decorate}]
(-.875,-.5) to[out=90, in =180] (0,.25);
\node at (-1.1,-.7) {\small$a$};
\end{tikzpicture}}
\quad  \quad
\hackcenter{ \begin{tikzpicture} [rotate=180, scale=.75, decoration={markings,
                        mark=at position 0.6 with {\arrow{<}};    }]
\draw[very thick, postaction={decorate}] (0,0) -- (0,1);
\draw[very thick, postaction={decorate}] (.875,-.5) -- (0,0);
\node at (1.1,-.7) {\small$n-a$};
\node at (0,1.3) {\small$n$};
\draw[very thick, postaction={decorate}] (-.875,-.5) -- (0,0);
\node at (-1.2,-.7) {\small$a$};
\end{tikzpicture}}
\;\; = \;\;
\hackcenter{ \begin{tikzpicture} [rotate=180, scale=.75, decoration={markings,
                        mark=at position 0.6 with {\arrow{<}};    }]
\draw[very thick] (0,.25) -- (0,0);
\draw[very thick, postaction={decorate}]
(.875,-.5) to[out = 90, in =0] (0,.25);
\node at (1.1,-.7) {\small$n-a$};
\draw[very thick, postaction={decorate}]
(-.875,-.5) to[out=90, in =180] (0,.25);
\node at (-1.1,-.7) {\small$a$};
\end{tikzpicture}}
\end{equation}
and the remaining edges and loops labeled $0$ or $n$ should be deleted.

Note that the cap and cup diagrams coming from the duality $a^+$
with $a^-$ arising from the pivotal structure do not require tags.

The followings  are consequences of the relations among generators of
$\fsl_n$-webs:

\begin{align}
\label{eq.switch}
\xy
(0,0)*{
\begin{tikzpicture} [scale=.75, decoration={markings,
                        mark=at position 0.6 with {\arrow{>}};    }]
\draw[very thick, postaction={decorate}] (0,-.5) -- (0,.25);
\node at (0,-.8) {\small$a$};
\draw[very thick, postaction={decorate}] (0,1) -- (0,.25);
\node at (0,1.3) {\small$n-a$};
\draw[very thick] (0,.25) -- (.2,.25);
\end{tikzpicture}};
\endxy
& = (-1)^{a(n-a)}
\xy
(0,0)*{
\begin{tikzpicture} [scale=.75, decoration={markings,
                        mark=at position 0.6 with {\arrow{>}};    }]
\draw[very thick, postaction={decorate}] (0,.25) -- (0,-.5);
\node at (0,-.8) {\small$a$};
\draw[very thick, postaction={decorate}] (0,.25) -- (0,1);
\node at (0,1.3) {\small$n-a$};
\draw[very thick] (0,.25) -- (.2,.25);
\end{tikzpicture}};
\endxy
\\ \label{eq:cancel-tags}
\xy
(0,0)*{
\begin{tikzpicture} [scale=.75]
\draw[very thick, directed=.55] (0,-.5) -- (0,.25);
\node at (0,-.8) {\small$a$};
\node at (0,2.8) {\small$a$};
\draw[very thick, directed=.45] (0,1) -- (0,.25);
\draw[very thick,directed=.75] (0,1) -- (0,2.5);
\node at (-.9,1.0) {\small$n-a$};
\draw[very thick] (0,.25) -- (-.2,.25);
\draw[very thick] (0,1.5) -- (.2,1.5);
\end{tikzpicture}};
\endxy \;\; & = \quad
\xy
(0,0)*{
\begin{tikzpicture} [scale=.75]
\draw[very thick, directed=.55] (0,-.5) -- (0,2.5);
\node at (0,-.8) {\small$a$};
\node at (0,2.8) {\small$a$};
\end{tikzpicture}};
\endxy
\end{align}

\begin{remark}
The tags appearing in
$n$-web (which do not appear in \cite{MOY}) play an important role
keeping track of the fact that while $(V_{e_a})^*$ is isomorphic to
$V_{e_{n-a}}$, this isomorphism is not canonical. The tags in
$\fsl_n$-webs keep track of these isomorphisms and contribute
signs that would have otherwise been missed by wrongly identifying the
dual of $a^+$ with $(n-a)^+$.
\end{remark}

\subsection{An equivalence between $\nWeb$ and $\nRep$}

The main result of  \cite{CKM} is the construction of an equivalence,
which is a $\Qqn$-linear pivotal functor,
$$
\Gamma_n: \nWeb \to \nRep
$$
defined on objects by $\Gamma_n(a^+)=V_{e_a}$ and
$\Gamma_n(a^-)=(V_{e_a})^*$. The ribbon structure of $\nRep$ can be pulled
back to make $\nWeb$ a ribbon category. In particular, we have a braiding
$X_{a, b}: a \ot b \to b \ot a$ for any two objects
$a, b$ of $\nWeb$.
For simple objects $a,b \in [1,n-1]$ we use the diagrams with crossings
as in Figure \ref{fig:braiding} to denote the braiding $X_{a,b}$ its
inverse $X^{-1}_{b,a}$. The braiding allows us to introduce crossings in
diagrams representing morphisms of $\nWeb$.

\begin{figure}[!hptb]
\begin{center}
\[
X_{a,b} \;\; = \;\;
\xy
(0,0)*{
\begin{tikzpicture}
\draw[very thick, directed=1] (-.5,0) -- (.5,1.5);
\draw[very thick ] (.5,0) -- (.1,.6);
\draw[very thick, directed=1 ] (-.1,.9) -- (-.5,1.5);
\node at (-.55,-.3) {\small$a$};
\node at (.55,-.3) {\small$b$};
\node at (-.55,1.8) {\small$b$};
\node at (.55,1.8) {\small$a$};
\end{tikzpicture}};
\endxy
\qquad \quad
X_{b,a}^{-1} \;\; = \;\;
\xy
(0,0)*{
\begin{tikzpicture}
\draw[very thick, directed=1] (.5,0) -- (-.5,1.5);
\draw[very thick ] (-.5,0) -- (-.1,.6);
\draw[very thick, directed=1 ] (.1,.9) -- (.5,1.5);
\node at (.55,-.3) {\small$a$};
\node at (-.55,-.3) {\small$b$};
\node at (.55,1.8) {\small$b$};
\node at (-.55,1.8) {\small$a$};
\end{tikzpicture}};
\endxy
\]
\caption{The braiding $X_{a,b}$ (left) and its inverse $X^{-1}_{b,a}$.}
\label{fig:braiding}
\end{center}
\end{figure}
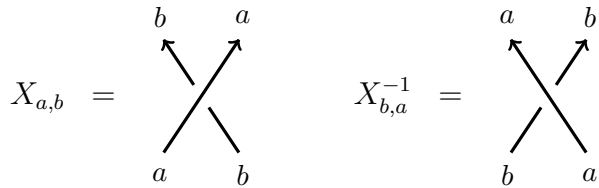

%

Suppose $D$ is a link diagram in the plane in general position with
respect to the height function, whose components are labeled by integers
in $[0,n-1]$. Then $D$ defines morphism in the category $\nWeb$ from
$\emptyset$ to $\emptyset$. Since $\Hom_{\nWeb}(\emptyset, \emptyset)=\Qqn$,
the morphism $D$ is determined by the scalar $\ev(D)\in\Qqn$.
The equivalence $\Gamma_n$ shows that this scalar $\ev(D)$ is equal to the
invariant $J_L^{\fsl_n}(e_{a_1}, \dots, e_{a_k})$, i.e.
\be
\label{eq.11}
J_L^{\fsl_n}(e_{a_1}, \dots, e_{a_k})= \ev(D),
\ee
where $L$ is the framed link whose blackboard diagram is $D$ and
$a_1,\dots, a_k$ are the labels of the components of $L$.

\subsection{The Ladder Category}
\label{sub.ladders}

We give the definition of ladder category $\Lad_m$, which a
diagrammatic presentation of Lusztig's idempotent form $\Uqgmd$ of the
quantum group $\Uqgm$. Typically, $\Uqgmd$ is regarded as a $\Qq$-algebra
where the unit is replaced by a system of mutually orthogonal idempotents
$1_{a }$ indexed by the weight lattice of $\fgl_m$. In \cite{CKM}, using
the quantum skew-Howe duality, they showed that there is a braided monoidal
functor from the ladder category to the category $\nWeb$. We explain  how to
use this result to calculate quantum $\Uqsn$-invariants of links using ladders.

\label{def.ladder}
{\em A ladder $Z$ with $m$ sides} is a uni-trivalent graph drawn in the
strip $\BR \times [0,1]$, with
\begin{itemize}
\item[(i)]
$m$ parallel vertical lines running from the bottom  line to the top
line of the strip, oriented upwards,
\item[(ii)]
some number of oriented horizontal  lines in the interior of the
strip $\BR \times [0,1]$, called steps, connecting adjacent sides,
\item[(iii)]
a labeling of each interval (steps or segments of sides) by integers,
such that the signed sum of the labels  at each trivalent vertex is zero.
Here the sign of each incoming vertex is positive, and the sign of each
outgoing vertex is negative.
\end{itemize}

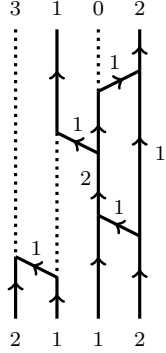
\begin{figure}[!hptb]
\begin{center}
\begin{tikzpicture} [rotate=-90,
decoration={markings, mark=at position 0.6 with {\arrow{>}}; },scale=.55]
\draw[very thick, dotted] (4.5,1) -- (8,1);
\draw[very thick, dotted] (2,0) -- (7.5,0);
\draw[very thick, dotted] (2,2) -- (3.5,2);
\draw[very thick, postaction={decorate}] (9,3) -- (7,3);
\draw[very thick, postaction={decorate}] (7,3) -- (3,3);
\draw[very thick, postaction={decorate}] (3,3) -- (2,3);
\draw[very thick, postaction={decorate}] (7,3) -- (6.5,2);
\draw[very thick, postaction={decorate}] (3.5,2) -- (3,3);
\draw[very thick, postaction={decorate}] (9,2) -- (6.5,2);
\draw[very thick, postaction={decorate}] (6.5,2) -- (5.5,2);
\draw[very thick, postaction={decorate}] (5.5,2) -- (3.5,2);
\draw[very thick, postaction={decorate}] (5,2) -- (4.5,1);
\draw[very thick, postaction={decorate}] (9,1) -- (8,1);
\draw[very thick, postaction={decorate}] (4.5,1) -- (2,1);
\draw[very thick, postaction={decorate}] (8,1) -- (7.5,0);
\draw[very thick, postaction={decorate}] (9,0) -- (7.5,0);
\node at (9.5,3) {$\scriptstyle 2$};
\node at (9.5,2) {$\scriptstyle 1$};
\node at (9.5,1) {$\scriptstyle 1$};
\node at (9.5,0) {$\scriptstyle 2$};
\node at (1.5,3) {$\scriptstyle 2$};
\node at (1.5,2) {$\scriptstyle 0$};
\node at (1.5,1) {$\scriptstyle 1$};
\node at (1.5,0) {$\scriptstyle 3$};
\node at (7.3,.5) {$\scriptstyle 1$};
\node at (6.3,2.5) {$\scriptstyle 1$};
\node at (5.6,1.7) {$\scriptstyle 2$};
\node at (4.3,1.5) {$\scriptstyle 1$};
\node at (2.8,2.4) {$\scriptstyle 1$};
\node at (5,3.5) {$\scriptstyle 1$};
\end{tikzpicture}
\caption{A morphism in $\Lad_3$.}
\label{fig.samplelad}
\end{center}
\end{figure}

Let $\partial_-Z$ (resp. $\partial_+Z$) be the sequence of labels appearing
on the bottom (resp. top) edge of the strip. Then
$\partial_-Z, \partial_+Z \in\BZ^m$ are considered as weights of $\Uqgm$.

The category $\Lad_m$ is the $\Qq$-linear category whose set of
objects is $\BZ^m$. Given two objects $a, b$, the morphisms
$\Hom_{\Lad_m}(a,b)$ is the set of all $\Qq$-linear combinations of ladders
$Z$ with $m$ sides such that $\partial_-Z=b, \partial_+Z=a$, subject to
the relations described in Equations \eqref{eq.com1}-\eqref{eq.com22} below.

Composition of morphisms is given by vertical concatenation of ladders.
Note that $\Lad_m$ does not have dual objects and hence is not pivotal.

\no{
Some authors use the notation $1_a \Lad_m 1_b$ for the set $\Hom_{\Lad_m}(a,b)$
where $1_a$ denotes the identity morphism of an object $a$, but we will not
use this notation here.
}

For an object $a=(a_1,\dots,a_m)$ of $\Lad_m$ and for $i$
such that $0\le i\le m-1$, and $r\in \BN$ let $\Eir1_a$ and $\Fir 1_a$
denote the following ladders:
\[
E_i^{(r)}1_{a } \;\; := \;\;   \xy
(0,0)*{
\begin{tikzpicture}[scale=.4]
\node at (-6,0) { $\dots$};
\node at (6,0) { $\dots$};
    \draw [very thick, directed=.55] (-4,-2) to (-4,2);
    \draw [very thick, directed=.55] (4,-2) to (4,2);
	\draw [very thick, directed=.55] (2,-.5) to (2,2);
	\draw [very thick, directed=.55] (2,-.5) to (-2,.5);
	\draw [very thick, directed=.55] (-2,-2) to (-2,.5);
	\draw [very thick, directed=.55] (-2,.5) to (-2,2);
	\draw [very thick, directed=.55] (2,-2) to (2,-.5);
\node at (-4,-2.55) {\tiny $a_{i-1}$};	
\node at (-2,-2.55) {\tiny $a_i$};	
\node at (2,-2.55) {\tiny $a_{i+1}$};
\node at (4,-2.55) {\tiny $a_{i+2}$};
	\node at (-2,2.5) {\tiny $a_i+r$};
	\node at (2,2.5) {\tiny $a_{i+1}-r$};
	\node at (0,0.75) {\tiny $r$};
\end{tikzpicture}
};
\endxy \in \Hom_{\Lad_m}(a,a+r \alpha_i)
\]

\[
F_i^{(r)}1_{a } \;\; := \;\;  \xy
(0,0)*{
\begin{tikzpicture}[scale=.4]
\node at (-6,0) { $\dots$};
\node at (6,0) { $\dots$};
    \draw [very thick, directed=.55] (-4,-2) to (-4,2);
    \draw [very thick, directed=.55] (4,-2) to (4,2);
	\draw [very thick, directed=.55] (2,.5) to (2,2);
	\draw [very thick, directed=.55] (-2,-.5) to (2,.5);
	\draw [very thick, directed=.55] (-2,-2) to (-2,-.5);
	\draw [very thick, directed=.55] (-2,-.5) to (-2,2);
	\draw [very thick, directed=.55] (2,-2) to (2,.5);
\node at (-4,-2.55) {\tiny $a_{i-1}$};	
\node at (-2,-2.55) {\tiny $a_i$};	
\node at (2,-2.55) {\tiny $a_{i+1}$};
\node at (4,-2.55) {\tiny $a_{i+2}$};
	\node at (-2,2.5) {\tiny $a_i- r$};
	\node at (2,2.5) {\tiny $a_{i+1}+r$};
	\node at (0,0.75) {\tiny $r$};
\end{tikzpicture}
};
\endxy
\in \Hom_{\Lad_m}(a,a-r \alpha_i)
\]

Here and in what follows, we draw the steps of a ladder
using slightly slanted lines instead of horizontal lines such that
the orientation of the step is upwards. With this convention we do not have
to mark the orientation in a ladder diagram, since all segments are oriented
upwards.

Comparing the sequence at the end of these ladders it is clear that
\begin{equation} \label{eq.com0}
E_i^{(r)}1_{a } = 1_{a+r\alpha_i}E_i^{(r)} = 1_{a+r\alpha_i}E_i^{(r)}1_{a}, \qquad
F_i^{(r)}1_{a } = 1_{a-r\alpha_i}F_i^{(r)} = 1_{a-r\alpha_i}F_i^{(r)}1_{a},
\end{equation}
When the specific weight is clear we will write $E_i$ instead of
$E_i 1_a$ and $F_i$ instead of $F_i 1_a$. For example,
$F^{(r)}_i E_j^{(s)} 1_a \ \text{ means }
\ F^{(r)}_i 1_{a + s \al_j} E_j^{(s)} 1_a.$

With this convention, the relations of the morphisms of $\Lad_m$ are given by

\begin{subequations}
\begin{align}
\label{eq.com1}
E_i^{(r)}F_i^{(s)}1_{a } &= \sum_{t=0}^{\min(r,s)}
{\langle a ,\alpha_i\ra +r -s \brack t} F_i^{(s-t)}E_i^{(r-t)}1_{a } &
\\ \label{eq.com2}
E_i^{(r)}F_j^{(s)}1_{a } &= F_j^{(s)}E_i^{(r)}1_{a } & \text{if $i\neq j$}
\\ \label{eq.com3}
E_i^{(r)}E_j^{(s)}1_{a } &= E_j^{(s)}E_i^{(r)}1_{a } &
\text{if $|i-j|>1$, likewise for $F$'s}
\\ \label{eq.com4}
E_i^{(s)}E_i^{(r)}1_{a } &= {r+s \brack r} E_i^{(r+s)}1_{a } &
\text{and  likewise for $F$'s}
\\ \label{eq.com22}
E_iE_jE_i1_{a } &= (E_i^{(2)}E_j + E_jE_i^{(2)})1_{a } &
\text{if $|i-j|=1$, likewise for $F$'s.}
\end{align}
\end{subequations}
for all $r, s \in \BN$, $0\le i \le m-1$ and $a \in \BZ^m$.

Recall $\langle a ,\al_i \ra= a_i - a_{i+1}$ is the standard inner product
on $\BZ^{m }$, and the quantum integers, factorial and binomial coefficients
are defined by

\begin{subequations}
\begin{align}
\label{eq.int}
[r] &= \frac{q^{r}-q^{-r}}{q-q^{-1}}, \qquad r \in \BZ \\
\label{eq.fac}
[r]! &= \prod_{k=1}^r [k], \qquad r \geq 0 \\
\label{eq.qbinom}
{r \brack s} &=
\begin{cases}
\frac{\prod_{k=r-s+1}^r[k]}{[s]!} & r,s \in \BZ, \quad s \geq 0 \\
0 & s < 0.
\end{cases}
\end{align}
\end{subequations}

\begin{remark}
\label{rem.idempotent}
If $k$ is a field and $\calC$ is a $k$-linear category, it gives
rise to an algebra $A(\calC)$ whose underlying vector space is the direct
sum of all Hom spaces $\oplus_{a, b} \Hom(a, b)$. The product of
$x \in \Hom(b,a)$ and $y \in \Hom(b',a')$ defined to be zero unless $b=a'$,
in which case the product is defined to be the composite $xy$.
$A(\calC)$ is a $k$-algebra without unit, in general. Since the
relations \eqref{eq.com1}--\eqref{eq.com22} are the defining relations
Lusztig's idempotent algebra $\dot{\mathrm{U}}_q(\mathfrak{gl}_m)$,
$A(\Lad_m)\cong\dot{\mathrm{U}}_q(\mathfrak{gl}_m)$.
\end{remark}

\subsection{The Schur quotient, the highest weight $\vartheta$,
and evaluation}

Fix positive integers $m$ and $n$. The Schur quotient $\Ladnm$ is defined to be
the $\Qqn$-linear category with set of objects all
$a=(a_1,\dots,a_m)\in \BZ^m$ such that $a \in [0,n]^m$, i.e.
$0\le a_i \le n$ for all $i$.
The algebra of morphisms of $\Ladnm$ is the quotient of the algebra
of morphisms of $\Lad_m$, with ground field extended to $\Qqn$, by the
two-sided ideal generated by all  $1_a$ with $a \not \in [0,n]^m$.

For example, $\Eir 1_a$ is always 0 in $\Ladnm$ when $r >n$.

Let
\be
\label{eq.vartheta}
\tnm:=(n^m,0^m)\in \BZ^{2m}
\ee
often abbreviated by $\vartheta$. Considered as an object of $n\Lad_{2m}$,
$\vartheta$ is a {\em highest weight element for $n\Lad_{2m}$}, in the
sense  that for every $i=1,\dots,2m-1$, we have

\begin{subequations}
\begin{align}
\label{eq.Ei}
E_i 1_{\vartheta} =0= 1_{\vartheta} F_i 1_{\vartheta +\al_i} =0.
\end{align}
\end{subequations}
This is because $\vartheta+\alpha_i$ has entries outside $[0,n]$. It follows
that the algebra of endomorphisms of $\vartheta$ is isomorphic to the
ground field $\Qqn$. In other words, we have an evaluation map
\begin{equation}
\label{eq.ev0}
\ev_n: \Hom_{\Ladnm}(\vartheta, \vartheta) \simeq \Qqn,\qquad
x=\ev(x) 1_\vartheta.
\end{equation}

\subsection{Braiding for ladders}
\label{sec.braid2ladder}

The category $\Ladnm$ does not have a tensor product and hence is not a
monoidal category. However, $\nLad:= \bigoplus_{m=1}^\infty \Ladnm$ is monoidal.  This category doesn't have duals since all webs are directed upwards.  But
it is a braided monoidal category, as follows.
The objects of $\nLad$ are sequences $a=(a_1,\dots,a_m)$ of integers
$a_i \in [0,n]$.
Given two objects $a=(a_1,\dots,a_m)$ and $b=(b_1,\dots,b_p)$,
$\Hom_{\nLad}(a,b)=\Hom_{\Ladnm}(a,b)$ if $p=m$ and $0$ otherwise.

The tensor product of objects $a \ot b$ is the horizontal concatenation
of $a$ and $b$ from left to right, and similarly for morphisms.

In \cite[Section 6]{CKM} it is shown that $\nLad$ admits a braided monoidal
category structure, i.e. it has a braiding, which is a system of natural
isomorphisms $X_{a, b}: a \ot b \to b \ot a$ satisfying the
hexagon equations \cite{RT,Tu2}.
The braiding for $\nLad$ is constructed using Lusztig's braid elements
\cite{Lu}.

We also use the diagrams with crossings in Figure \ref{fig:braiding} to
denote the braiding $X_{a,b}$ and its inverse $ X^{-1}_{a,b}$ in the category
$\Ladnm$.

When $\beta$ is a braid on $m$ strands and $a =(a_1,\dots,a_m)\in \BZ^m$,
let $\beta 1_a  \in \Hom_{\Ladnm}(\beta(a),a)$ be the morphism described in
Figure \ref{fig:braiding2}. Here $\beta(a)$ is obtained from $a$ by
applying the permutation corresponding to the braid $\beta$. For example,
$\sigma_i 1_a$ and $\sigma^{-1}_i 1_a$, where $\sigma_i, \sigma^{-1}_i $
are the $i$-th  standard braid generator and its inverse, are depicted in
Figure \ref{fig:braiding2}.


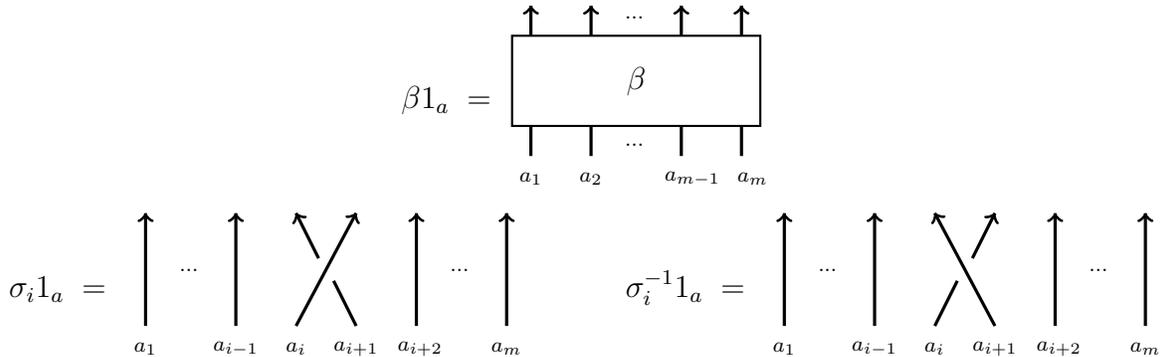
\begin{figure}[!hptb]
\begin{center}
\[
\beta1_a \; = \;
\xy
(0,0)*{
\begin{tikzpicture}
\draw[very thick, directed=1] (-1.4,0) -- (-1.4,2);
\draw[very thick, directed=1] (-.6,0) -- (-.6,2);
\draw[very thick, directed=1] (1.4,0) -- (1.4,2);
\draw[very thick, directed=1] (.6,0) -- (.6,2);
\draw[thick, fill=white, draw=black] (-1.65,.4) -- (-1.65,1.6) --
(1.65,1.6) -- (1.65,.4) -- cycle;
\node at (-1.4,-.3) {$\scs a_{1}$};
\node at (-.6,-.3) {$\scs a_{2}$};
\node at (.75,-.3) {$\scs a_{m-1}$};
\node at (1.55,-.3) {$\scs a_{m}$};
\node at (0,.15) {$\scs\dots$};
\node at ( 0,1.85) {$\scs\dots$};
\node at ( 0,1) {$\beta$};
\end{tikzpicture}};
\endxy
\]
\[
\sigma_i 1_{a} \; = \;
\xy
(0,0)*{
\begin{tikzpicture}
\draw[very thick, directed=1] (-2.4,0) -- (-2.4,1.5);
\draw[very thick, directed=1] (-1.2,0) -- (-1.2,1.5);
\draw[very thick, directed=1] (-.4,0) -- (.4,1.5);
\draw[very thick ] (.4,0) -- (.1,.6);
\draw[very thick, directed=1 ] (-.1,.9) -- (-.4,1.5);
\draw[very thick, directed=1] (1.2,0) -- (1.2,1.5);
\draw[very thick, directed=1] (2.4,0) -- (2.4,1.5);
\node at (-2.4,-.3) {$\scs a_{1}$};
\node at (-1.2,-.3) {$\scs a_{i-1}$};
\node at (-.4,-.3) {$\scs a_i$};
\node at (.4,-.3) {$\scs a_{i+1}$};
\node at (2.4,-.3) {$\scs a_{m}$};
\node at (1.25,-.3) {$\scs a_{i+2}$};
\node at (-1.8,.75) {$\scs\dots$};
\node at ( 1.8,.75) {$\scs\dots$};
\end{tikzpicture}};
\endxy
\qquad \quad
\sigma_i^{-1}1_a \; = \;
\xy
(0,0)*{
\begin{tikzpicture}
\draw[very thick, directed=1] (-2.4,0) -- (-2.4,1.5);
\draw[very thick, directed=1] (-1.2,0) -- (-1.2,1.5);
\draw[very thick, directed=1] (.4,0) -- (-.4,1.5);
\draw[very thick ] (-.4,0) -- (-.1,.6);
\draw[very thick, directed=1 ] (.1,.9) -- (.4,1.5);
\draw[very thick, directed=1] (1.2,0) -- (1.2,1.5);
\draw[very thick, directed=1] (2.4,0) -- (2.4,1.5);
\node at (-2.4,-.3) {$\scs a_{1}$};
\node at (-1.2,-.3) {$\scs a_{i-1}$};
\node at (-.4,-.3) {$\scs a_i$};
\node at (.4,-.3) {$\scs a_{i+1}$};
\node at (2.4,-.3) {$\scs a_{m}$};
\node at (1.25,-.3) {$\scs a_{i+2}$};
\node at (-1.8,.75) {$\scs\dots$};
\node at ( 1.8,.75) {$\scs\dots$};
\end{tikzpicture}};
\endxy
\]
\caption{The morphisms $\beta 1_a$,  $\sigma_i 1_a$, and
$\sigma^{-1}_i 1_a$}
\label{fig:braiding2}
\end{center}
\end{figure}

Then $\sigma_i^{\pm1} 1_a \in \Hom_{\Ladnm}(\sigma_i(a),a)$. We record here
the formula for the braidings from \cite{CKM}:
\begin{align}
\label{eq.tau1}
\sigma_i 1_a & =   \;\; (-1)^{a_i +a_ia_{i+1}}q^{a_i-\frac{a_i a_{i+1}}{n}}
  \sum_{
  \xy (0,1)*{\scriptstyle r,s \geq 0};
(0,-1.4)*{\scriptstyle s-r=a_i-a_{i+1}}; \endxy }
  (-q)^{-s} E_i^{(r)} \, F_i^{(s)} 1_{a}\\
\label{eq.tau2}
\sigma^{-1}_i 1_a
  &= \;\; (-1)^{a_i +a_ia_{i+1}}q^{-a_i+\frac{a_i a_{i+1}}{n}}
  \sum_{
  \xy (0,1)*{\scriptstyle r,s \geq 0};
(0,-1.4)*{\scriptstyle s-r=a_i-a_{i+1}}; \endxy }
  (-q)^{s} E_i^{(r)} \, F_i^{(s)} 1_{a}
\end{align}
Note that the right hand sides are finite sums, since $\Fir$ and $\Eir$
are 0 for $r > n$. Also $\sigma^{-1}_i 1_a$ is obtained from
$\sigma_i 1_a$ by the involution $q \to q^{-1}$.

\no{
In picture, this means,

\begin{equation}
\label{eq.web-braid}
  X_{a,b} \;\; := \;\; (-1)^{a+ab}q^{a-\frac{ab}{n}}
  \sum_{
  \xy (0,1)*{\scriptstyle r,s \geq 0};
(0,-1.4)*{\scriptstyle s-r=a-b}; \endxy }
  (-q)^{-s}
  \xy
(0,0)*{
\begin{tikzpicture}[scale=.35]
    \node[very thick] at (-5,-2) { $\dots$};
    \node[very thick] at (5,-2) { $\dots$};
	\draw [very thick, directed=.55] (-2,-4) to (-2,-2);
	\draw [very thick, directed=1] (-2,-2) to (-2,0.25);
	\draw [very thick, directed=.55] (2,-4) to (2,-2);
	\draw [very thick, directed=1] (2,-2) to (2,0.25);
	\draw [very thick, directed=.55] (-2,-2) to (2,-2);
	\draw [very thick] (-2,0.25) to (-2,2);
	\draw [very thick, directed=.55] (-2,2) to (-2,4);
	\draw [very thick] (2,0.25) to (2,2);
	\draw [very thick, directed=.55] (2,2) to (2,4);
	\draw [very thick, rdirected=.55] (-2,2) to (2,2);
	\node at (-2,-4.5) {\tiny $a_i$};
	\node at (2,-4.5) {\tiny $a_{i+1}$};
	\node at (-2,4.5) {\tiny $a_{i+1}$};
	\node at (2,4.5) {\tiny $a_i$};
	\node at (-3.7,0) {\tiny $a_i-s$};
	\node at (4,0) {\tiny $a_{i+1}+s$};
	\node at (0,-1.25) {\tiny $s$};
	\node at (0,2.75) {\tiny $r$};
\end{tikzpicture}
};
\endxy
\end{equation}
}

\begin{remark}
\label{rem.Ttriple} Originally Luzstig~\cite[5.2.1]{Lu} defined the
braiding and its inverses using triple product
formulas. The simplification of
Lusztig's formula's to double products in
 equations~\eqref{eq.tau1}--\eqref{eq.tau2} was first observed for
$q=1$ by Chuang and Rouquier \cite{CR}.  For general $q$, a proof of
this simplification  can be found in \cite[Lem.6.1.1]{CKM}.

\no{

of a triple exponential

Typically the action of $T_i$ is given by a triple product
formula. Indeed, Luzstig~\cite[5.2.1]{Lu} gives the following form
of a triple exponential
\begin{subequations}
\begin{align}
\label{eq:tau_prime}
\tau_{i,e}'1_{\l} &= \sum_{a,b,c; a-b+c=\lambda_i}
(-1)^b q^{e(-ac+b)}F_i^{(a)}E_i^{(b)}F_i^{(c)} 1_{\l},   \\
\label{eq:tau_dprime}
\tau_{i,e}''1_{\l} &= \sum_{a,b,c; -a+b-c=\lambda_i}
(-1)^b q^{e(-ac+b)}E_i^{(a)}F_i^{(b)}E_i^{(c)} 1_{\l}.
\end{align}
\end{subequations}
The simplification of equations~\eqref{eq:tau_prime}--\eqref{eq:tau_dprime}
to equations~\eqref{eq.tau+}--\eqref{eq.tau-} was first observed for
$q=1$ by Chuang and Rouquier \cite{CR}.  For general $q$, a proof of
this equivalence can be found in \cite[Lem.6.1.1]{CKM}.
}
\end{remark}

\subsection{From ladders to webs}
\label{sub.webs2ladders}

In \cite[Section 5]{CKM} it is proved that there is a $\Qqn$-linear functor
$$
\Psi_{n,m}: \Ladnm \to \nWeb
$$
defined as follows. For an object $a=(a_1,\dots,a_m)$ of $\Ladnm$,
$\Psi_{n,m}(a)$ is obtained from $a$ by deleting $0$s and $n$s from $a$
and converting $k$ to $k^+$. For a morphism $f$ of $\Ladnm$ which is a ladder,
$\Psi_{n,m}(f)$ is the same $f$ considered as an $n$-web, using the convention
about labelings 0 and $n$.
This means edges connected to the label $0$ should be deleted from the
diagrams and those connected to the label $n$ should be truncated to the
``tags'' depicted in the last two diagrams in equation~\eqref{webgen} as
explained in \eqref{eq_tag_cup}. The existence of $\Psi_{n,m}$ is a
consequence of the quantum skew-Howe duality.

The functors $\Psi_{n,m}: \Ladnm \to \nWeb$, with all $m$, piece together to
give a functor $\Psi_n: \nLad \to \nWeb$. By Theorem
\cite[Theorem 6.2.1]{CKM}, $\Psi_n$ is a braided monoidal functor.

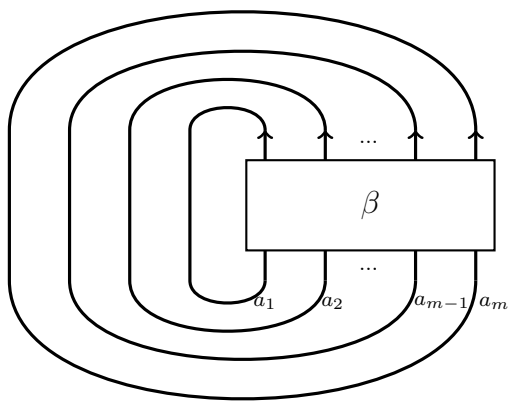
\begin{figure}[!hptb]
\begin{center}
\[
\xy
(0,0)*{
\begin{tikzpicture}
\draw[very thick, directed=1] (-1.4,0) -- (-1.4,2);
\draw[very thick, directed=1] (-.6,0) -- (-.6,2);
\draw[very thick, directed=1] (.6,0) -- (.6,2);
\draw[very thick, directed=1] (1.4,0) -- (1.4,2);
\draw[very thick ] (-1.4,2) .. controls ++(0,.4) and ++(0,.4) .. (-2.4,2)
    -- (-2.4,0) .. controls ++(0,-.4) and ++(0,-.4) .. (-1.4,0);
\draw[very thick ] (-.6,2) .. controls ++(0,.9) and ++(0,.9) .. (-3.2,2)
               -- (-3.2,0) .. controls ++(0,-.9) and ++(0,-.9) .. (-.6,0);
\draw[very thick ] (.6,2) .. controls ++(0,1.4) and ++(0,1.4) .. (-4,2)
                -- (-4,0) .. controls ++(0,-1.4) and ++(0,-1.4) .. (.6,0);
\draw[very thick ] (1.4,2) .. controls ++(0,2.1) and ++(0,2.1) .. (-4.8,2)
                -- (-4.8,0) .. controls ++(0,-2.1) and ++(0,-2.1) ..
(1.4,0);
\draw[thick, fill=white, draw=black] (-1.65,.4) -- (-1.65,1.6) --
(1.65,1.6) -- (1.65,.4) -- cycle;
\node at (-1.4,-.3) {$\scs a_{1}$};
\node at (-.5,-.3) {$\scs a_{2}$};
\node at (.95,-.3) {$\scs a_{m-1}$};
\node at (1.65,-.3) {$\scs a_{m}$};
\node at (0,.15) {$\scs\dots$};
\node at ( 0,1.85) {$\scs\dots$};
\node at ( 0,1) {$\beta$};
\end{tikzpicture}};
\endxy
\]
\caption{The standard closure of a braid $\beta$ with 4 strands.}
\label{fig:cl.braid}
\end{center}
\end{figure}


Suppose $\beta$ is a braid on $m$ strands. We view $\beta$ as a diagram
with crossings in the standard plane with strands oriented upwards. Let
$\cl(\beta)$ be the link diagram obtained by closing $\beta$ in the
standard way (see Figure \ref{fig:cl.braid}) and $L=L(\beta)$ be the
corresponding framed link. Assume $L$ has $k$ ordered components  which
are labeled by integers $a_1,\dots, a_k \in [1,n-1]$. Let
$a=(a_1,\dots,a_k)$. Let $b_1,\dots,b_m$ be the induced labeling of
strands of $\beta$ from left to right (at the bottom of $\beta$). Of course
each $b_i$ is one of the $a_j$'s.

Let $\Lcl(\beta,a)$, called the {\em ladder closure of $\beta$},
be the endomorphism of $\tnm$  in the category $n\Lad_{2m}$  given by the
ladder described  in Figure \ref{fig:Lclbraid}. Here the labels of the
strands of the braids are $b_1,\dots,b_m$ which are determined by the labels
$a_1,\dots,a_k$ of the link $L$. All the dashed vertical lines of the $m$
left sides are labeled by $n$ while all the dashed vertical lines of the
$m$ right sides are labeled by $0$. Then the remaining labels are
uniquely determined by the rule that the signed sum at every trivalent
vertex is $0$.

\begin{figure}[!hptb]
\begin{center}
\includegraphics[height=0.60\textheight]{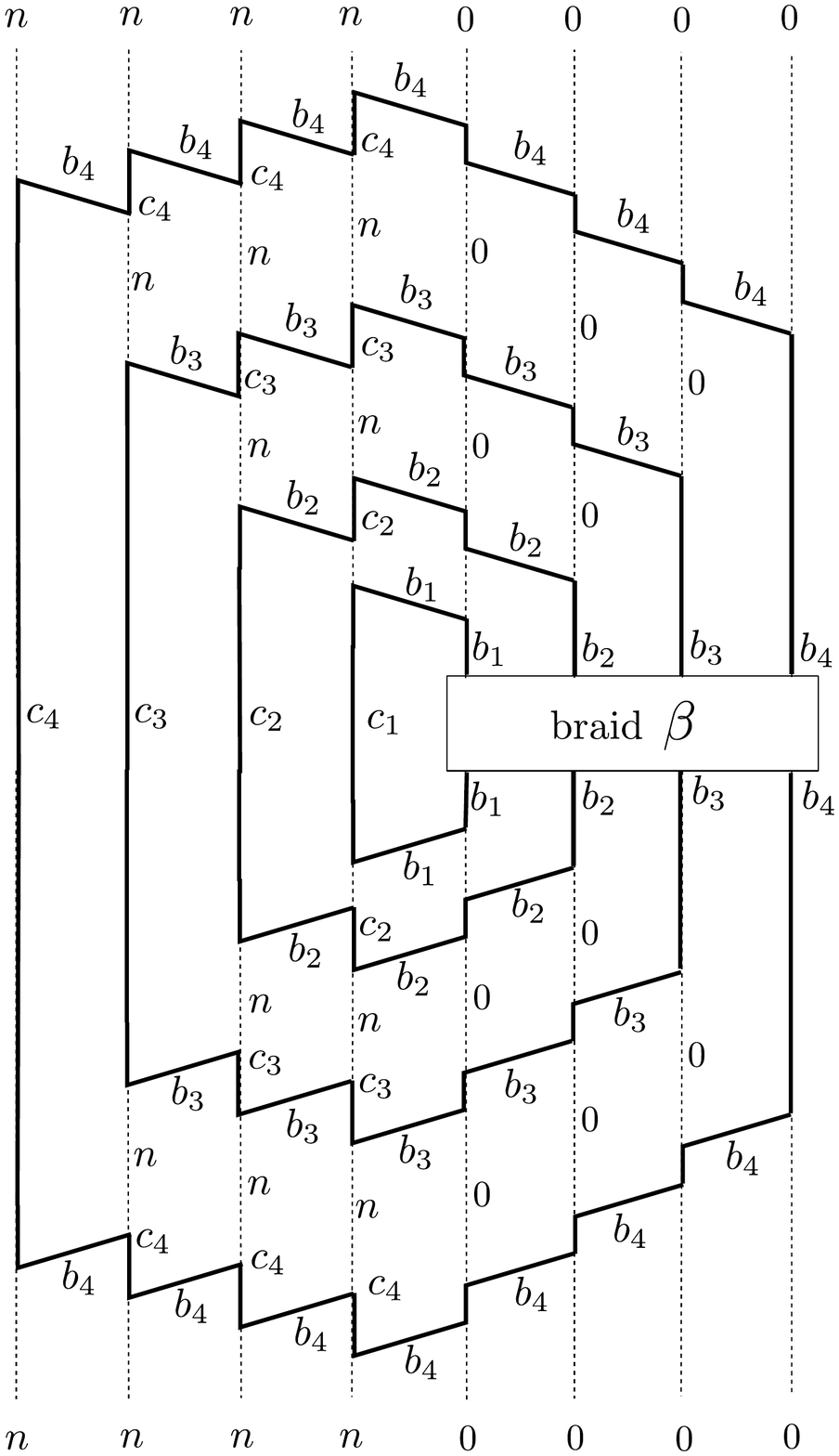}
\caption{The ladder closure of a braid $\beta$ with 4 strands, with labels.
Here $c_i= n-b_i$.}
\label{fig:Lclbraid}
\end{center}
\end{figure}

\begin{proposition}
\label{r.Ladder}
We have
$$
\ev( \Lcl(\beta,a)) = J_L(e_{a_1},\dots,e_{a_k}).
$$
\end{proposition}

\begin{proof}
Let $L$ denote the closure of $\beta$. $L$ is a link colored by $a$.
Identities \eqref{eq.switch} and  \eqref{eq:cancel-tags} show that
$$
\Psi_n(L,a)=\cl(\beta,a).
$$
Since $\Psi_n$ is a $\Qqn$-linear braided functor, we have
\begin{align*}
\ev(\cl(\beta,a)) & = \ev(\Psi_n(L,a))\\
& = J_L(e_{a_1},\dots,e_{a_k}).
\end{align*}
where the second identity follows from \eqref{eq.11}.
\end{proof}


\section{Introducing the variable $x=q^n$}
\label{sec.x}

Proposition \ref{r.Ladder} allows one to calculate the quantum
$\fsl_n$-invariant of a link $L$ for each fixed $n \geq 2$. In this section,
we introduce an algebra that allows us to unify the quantum $\fsl_n$
invariants of links into Laurent polynomials of a variable $x=q^n$.

\subsection{Free associative algebra on $E_i,F_j$}
\label{sub.free}
Let
\be
\label{eq.calX}
\calX_{m}=\{ E_1,\dots,E_{m-1}, F_1,\dots,F_{m-1}\}
\ee
and $\Ass_{m}$ be the free
associative $\BQ(q)$-algebra generated by $\calX_m$. For $i=1,\dots, m-1$,
define the divided powers by
$$
E_i^{(r)}:= E_i^r/[r]! \in \Ass_m , \quad F_i^{(r)}:= F_i^r/[r]! \in \Ass_m ,
$$
where $[r]!$ is given by equation \eqref{eq.int}.
A $\Qq$-basis of $\Ass_m$ can be described as follows. For
$Y=(Y_1,\dots,Y_r)\in (\calX_m)^r$ and
$k=(k_1,\dots,k_r)\in \BN^r$ define
$$
Y^{(k)} := Y_1^{(k_1)} Y_2^{(k_2)} \dots Y_r^{(k_r)}.
$$
Then the set of all $Y^{(k)}$, where $Y_{i}\neq Y_{i+1}$ and $k_i \ge 1$,
along with $k=\emptyset$, is a $\Qq$-basis of $\Ass_m$.

Note that for each $a \in \BZ^m$, $Y^{(k)} 1_a$ is a morphism in the
category $\Ladnm$. For $a,b\in \BZ^{m }$ and $n>1$, define the
$\Qq$-linear map
$$
\pabn: \Ass_m \to \Hom_{\Ladnm}(a,b), \quad
\pabn( Y^{(k)} ) =  1_{a} Y^{(k)} 1_{b}.
$$
The algebra $\Ass_{m}$ admits a natural $\BZ^m$-grading, called {\em weight},
defined by
$$
\w (F_i)=-\alpha_i, \qquad \w (E_i) = \alpha_i.
$$
Observe that $\pabn (Y^{(k)})=0$ unless $a = b + \w(Y^{(k)})$.

Let $I_k$ be the two-sided ideal of $\Ass$ generated by
$E_i^{(r)}, F_i^{(r)}$, with $i=1,\dots, m -1$ and $ r \ge k$. It is clear
that $I_{k+1} \subset I_k$. Let $\hAss_m$ be the completion of
$\Ass_m$ with respect to the nested sequence of ideals $I_k$.
\no{Every element
$x \in\hAss_m$ has a presentation of the form
\be
x = \sum_{k=0}^\infty x_k,
\notag
\ee
where $x_k \in I_k$.}
Since $p^n_{a,b}(I_k)=0$ if $k > n$,
we can extend  $p^n_{a,b}$  to a map, also denoted by $p^n_{a,b}$,
\be
\nonumber
p^n_{a,b}:\hAss_m\to \Hom_{n\Lad_m}(a,b).
\ee

\subsection{Convention on negative powers}

The divided powers $E_i^{(r)}$ and $F_i^{(r)}$ are defined for non-negative
integers $r$. It is convenient to extend them to negative powers by the
following convention.
For $r<0$, $a \in \BZ^{m}$, and $s \in \BZ$ we use the following
convention
\begin{align*}
E_i^{(r)} & = F_i^{(r)}=0 \quad \text{ in }  \hAss_m\\
E_i^{(r)} 1_a  & = F_i^{(r)}1_a=0\quad \text{ in } \ \Lad_m
\end{align*}

With the above convention, Equations ~\eqref{eq.com1}, \eqref{eq.com2}, and
\eqref{eq.com4} can be rewritten in the following form: For all $r,s \in \BZ$
and $i\neq j$, we have the following identities in $\Lad_m$.

\begin{align}
\label{eq.com1a}
 E_i^{(r)}F_i^{(s)}1_{a } &=
\sum_{t\in \BZ}
{\langle a ,\alpha_i\ra +r -s \brack t} F_i^{(s-t)}E_i^{(r-t)}1_{a }\\
\label{eq.com2a}
E_i^{(r)}F_j^{(s)}1_{a } &= F_j^{(s)}E_i^{(r)}1_{a }
\\ \label{eq.com4a}
E_i^{(s)}E_i^{(r)}1_{a } &= {r+s \brack r} E_i^{(r+s)}1_{a },
\quad F_i^{(s)}F_i^{(r)}1_{a } =  {r+s \brack r} F_i^{(r+s)}1_{a }.
\end{align}

\subsection{Evaluation}
\label{sub.evalution}

Fix positive integers $n,m$. Recall that $\vartheta$ given by
\eqref{eq.vartheta} is an object of $n\Lad_{2m}$, and recall the evaluation
map \eqref{eq.ev0}. This gives rise to an evaluation map

\begin{equation}
\label{eq.ev}
\ev_n:  \hAss_{2m}  \to \Qqn,\quad
\ev_n(x) := \ev_{n,m} \left ( p^n_{\vartheta,\vartheta}(x)  \right).
\end{equation}

Given a monomial $z$ in $E_i, F_j$ the element $\ev_n(z)$ can be
calculated by a simple algorithm moving each divided power in $E_i$ appearing
in $z$ to the right using equations \eqref{eq.com1a} and \eqref{eq.com2a}.
Note that we are not moving divided powers of  $E_i$ past divided powers of
$E_j$. Since the $E_i$'s annihilate the weight $1_\vartheta$,
all that remains after sliding all $E_i$'s to the right is a sum of
products of the quantum binomials produced from the application of
\eqref{eq.com1a}. For details see the example in Section \ref{sub.trefoil}
and Proposition \ref{prop.tech1}.

Suppose $Y=(Y_1,\dots,Y_k)\in (\calX_{2m})^k$ and
$b=(b_1,\dots,b_k)\in \BZ^k$. There is an easy case when
$\ev_n(Y^{(b)})=0$, namely when $1_\vartheta Y^{(b)} 1_\vartheta$
factors through a weight with a negative component. The weight of $Y^{(b)}$
is denoted by
$$
\w(Y^{(b)})=(\w_1(Y^{(b)}), \dots, \w_{2m}(Y^{(b)})) \in \BZ^{2m}.
$$
We say $Y^{(b)}$ has {\em negative} weight if $\w_j(Y^{(b)})<0$  for some $j$
with $m<j \le 2m$.
For an index $i, 1\le i \le k$ define the $i$-th tail $\Tail_i(Y,b)$ by
$$
\Tail_i(Y,b) = Y_i^{(b_i)}Y_{i+1}^{(b_{i+1})}  \dots Y_k^{(b_k)}.
$$
We say $(Y,b)$ is {\em tail-negative} if there is an index
$i, 1\le i \le k$, such that $\Tail_i(Y,b)$ has negative weight.

\begin{lemma}
\label{r.00}
Suppose $(Y,b)$ is tail-negative. Then $\ev_n(Y^{(b)})=0$ for all $n$.
\end{lemma}

\begin{proof}
Note that  $Y^{(b)} 1 _\vartheta$ factors through
$\Tail_i(Y,b) 1_\vartheta \in \Hom_{\Lad_{2m}^n}(\mu,\vartheta)$, where
$$
\mu = \w(\Tail_i(Y,b)) + \vartheta.
$$
Suppose $\w_j(\Tail_i(Y,b)) < 0$ for some $j >m$ and $1 \le i \le k$.
We have $\mu_j= \w(\Tail_i(Y,b)) + \vartheta_j = \w_j(\Tail_i(Y,b))<0$.
By definition, $\Tail_i(Y,b) 1_\vartheta =0$ in $\Lad_{2m}^n$. Hence
$Y^{(b)} 1 _\vartheta=0$ in $\Lad_{2m}^n$.
\end{proof}

The tail-negative condition can be characterized by the following function
\be
\label{eq.HH}
\sH(Y,b):= \prod_{j=m+1}^{2m} \prod_{i=1}^k \He(\w_j(\Tail_i(Y,b))).
\ee
where
\be
\label{eq.he}
\He(x)=\begin{cases}
1 & \text{if}\,\, x \geq 0 \\
0 & \text{if}\,\, x < 0 .
\end{cases}
\ee
denotes the Heaviside function. Note that
\be
\label{eq.tail}
\sH(Y,b)=
\begin{cases} 0 \quad & \text{if $(Y,b)$ is tail-negative}\\
1 & \text{otherwise.}
\end{cases}
\ee

\subsection{Braiding in $\hAss$}
\label{sub.braiding}

Suppose $a=(a_1,\dots,a_m)\in \BZ^m$ and $0\le i \le m-1$. Let

\begin{subequations}
\begin{align}
\label{eq.tau+}
 T_i (a) &= (-1)^{a_i+a_ia_{i+1}} q^{a_i}
\sum_{s\in \BZ} (-q)^{-s}  E_i^{(s+a_{i+1}-a_i)}F_i^{(s)} \in \hAss \\
 \label{eq.tau-}
 T_i^{-1}(a) &= (-1)^{a_i+a_ia_{i+1}} q^{-a_i}
\sum_{s\in \BZ} (-q)^{s}  E_i^{(s+a_{i+1}-a_i)}F_i^{(s)}  \in \hAss
\end{align}
\end{subequations}
Recall that we use the convention $\Eir=\Fir=0$ if $r<0$. Note that
$T_i^{-1}(a)$ is obtained from $ T_i(a)$ by the involution $q \to q^{-1}$.
From equation \eqref{eq.tau1} and \eqref{eq.tau2} it follows that
for $\ve=\pm 1$,

\be
\label{eq.match}
\sigma_i^\ve 1_a = q^{\ve \frac{ a_i a_{i+1}} {n}}T_i^{\ve}(a) 1_a
\quad \text{ in $\Ladnm$}.
\ee

\subsection{Special functions}
\label{sec.Homf}

Let $Y=(Y_1,\dots,Y_k)\in (\calX_m)^k$. A function $H: \BZ^r \to \hAss_{m}$
is called a $Y$-{\em special} if
\be
\label{eq.28}
H(a) =  \sum_{s \in \BZ^t} (-1)^{g_1(a,s)} q^{g_2(a,s)}
Y^{(f(a,s))},
\ee
where $g_1,g_2:\BZ^{r+t} \to \BZ$  are linear and  $ f: \BZ^{r+t} \to\BZ^k$
is affine such that $f(a,\cdot): \BZ^{t} \to\BZ^k$ is injective for
every $a \in \BZ^k$. The injectivity property ensures that the right
hand side of \eqref{eq.28} defines an element in $\hAss_m$.
The next lemma is easy to verify.

\begin{lemma}
\label{r.2}
\rm{(a)}
The function $T_i, T^{-1}_i:\BZ^m \to \hAss_m$ given by Equations
\eqref{eq.tau+} and \eqref{eq.tau-} are $(E_i, F_i)$-special.
\newline
\rm{(b)}
Suppose $f: \BZ^k \to \BZ^r$ is a linear function. Then the function
$H: \BZ^k \to \hAss_m$ given by
$$
H(a)= Y^{(f(a))}
$$
is $Y$-special.
\newline
\rm{(c)}
If $H'$ is $Y'$-special and $H''$ is $Y''$-special, then $H' H''$ is
$Y' \cup Y''$-special.
\end{lemma}

\begin{remark}
A function $g: \BZ^s \to \BZ$ is {\em quadratic} if there is a $s\times s$
matrix $A$ with integer entries such that $g(a) = a^T A a$.
If $g:\BZ^s \to \BZ$ is quadratic, then there is a linear $f:\BZ^s \to \BZ$
such that $(-1)^{g(a)}= (-1)^{f(a)}$ for all $a \in \BZ^s$.
Indeed,
suppose $g(a) = a^T A a$, where $A$ is a $s\times s$ matrix $A$ with
integer entries. Let $B$ be the $s\times s$ diagonal matrix defined by
$B_{ii}= A_{ii}$ and $f(a):= B a$. Then we have $(-1)^{g}= (-1)^f$.
\end{remark}

\subsection{Unifying the $\fsl_n$-link invariant}

For $a =(a_1,\dots,a_r)\in \BZ^r$ let $\|a \|_\infty$ be the usual norm
defined by $\|a \|_\infty= \max_i |a_i|$.

\begin{proposition}
\label{r.main0}
Suppose $L$ is framed oriented link in the 3-space with $r$ ordered
components which is the closure of a braid with $m$ strands.
Then there is $Y$-special function $H: \BZ^r \to \hAss_{2m}$
such that for all integers  $a_1, \dots, a_r\in [0,n-1]$, we have
\be
\label{eq.m0}
\tJ_L^{\fsl_n}(e_{a_1}, \dots, e_{a_r})= \ev_n (H(a_1,\dots, a_r)).
\ee
Moreover,  $(Y, f(a,s))$ is tail-negative whenever
$\|s \|_\infty > \|a \|_\infty$. Here $f(a,s)$ is the function
appearing in the presentation \eqref{eq.28} of $H$.
\end{proposition}

\begin{proof}
Let  $L$ be  the closure of a braid $\beta\in B_{2m}$ as in Figure
\ref{fig:cl.braid}, and $a=(a_1,\dots,a_r)\in \BN^r$. Here $\beta$ is a
braid which is non-trivial only  on the right $m$ strands, and has a
presentation
\be
\label{eq.1d}
\beta= \sigma_{i_1}^{\ve_1} \dots \sigma_{i_t} ^{\ve_t}, \quad
\text{$i_j \in \{m+1,\dots,2m-1\}$ and $\ve_j \in
\{\pm 1\}$ for $j=1,\dots,t,$}
\ee
where $\sigma_i$ is the $i$-th standard generator of the braid group (see
Figure \ref{fig:braiding}).

Let $\Lcl(\beta,a)$, with labels coming from the labels $a_i$ of $L$, be
the ladder closure of $\beta$ as in Figure \ref{fig:Lclbraid}. Recall that
the labelings $b_i$ of the strands of $\beta$ come from the labelings of $L$,
and $c_i=n-b_i$. Let $b=(b_1,\dots,b_m)$ and
$c =(c_m,c_{m-1}, \dots, c_1)$. Then each $b_i$ is one of  $(a_1,\dots, a_r)$.

The horizontal lines at the bottom and the top of the braid $\beta$
decomposes $\Lcl(\beta)$ into 3 morphisms in $\Lad^n_{2m}$:
$$
\Lcl(\beta,a)= \Capp_m(a)\,  (\beta 1_{c \ot b} ) \, \Cupp_m (a).
$$
Each part can be written in a form that does not depend on $n$.

Indeed, the lower morphism $\Cupp_m (a)$ is a composition of $E_i^{(b_j)}$ for
various $i,j$. Hence
\be
\nonumber
\Cupp_m (a) = \Cupam_m(a) 1_{\vartheta(n,m)},
\ee
where $\Cupam_m(a) \in \Ass_{2m}$ is the product of several $F_i^{(b_j)}$.
Explicitly,
\be
\nonumber
\Cupam_m(a) =
\rprod_{k\in [1,m]} \left[\left(\lprod_{i\in [1,k-1]}
 \left( F_{m+k}^{(b_k)} F_{m-k}^{(b_k)}\right)\right) F_m^{(b_k)}\right],
\ee
where
$$
\rprod_{i\in [1,k]} x_i := x_1 x_2 \dots x_k, \qquad
\lprod_{i\in [1,k]} x_i := x_k x_{k-1} \dots x_1.
$$
Then $\Cupam_m(a)$, as a function of $a$ is a special function,
see Lemma \ref{r.2}.

Similarly, the top morphism
\be
\nonumber
\Capp_m (a) =   \Capam_m(a)  1_{c \ot b}
= 1_{\vartheta(n,m)} \Capam_m(a) 1_{c \ot b}
\ee
is a special function. Explicitly,
\be
\nonumber
\Capam_m(a) =\lprod_{k\in [1,m]} \left[E_m^{(b_k)}  \rprod_{i\in [1,k-1]}
\left( E_{m+k}^{(b_k)} E_{m-k}^{(b_k)}\right) \right].
\ee

Now consider the middle morphism $\beta 1_{c \ot b}$. Using \eqref{eq.1d}
and \eqref{eq.match}, we have
\be
\nonumber
\beta 1_{c \ot b} = z_1(a) z_2(a) \dots z_t(a) 1_{c \ot b},
\ee
where
$$
z_j(a) =  T_{i_j}^{\ve_j}\left(\sigma_{i_j+1} \dots \sigma_{i_t}(c,b)\right).
$$
Using equation \eqref{eq.tau+} and \eqref{eq.tau-} for $T_i^{\pm1}(b)$,
we see that $z_j$ is a special function. Let
\be
\label{eq.Xve}
H(a)= \Capam_m(a)  \tau_\beta(a)  \Cupam_m(a)
=  \Capam_m(a) z_1(a) z_2(a) \dots z_t(a)  \Cupam_m(a) .
\ee
Then $H: \BZ^k \to \hAss_{2m}$, is a product of special functions,
hence is a special function (see Lemma \ref{r.2}). By \eqref{eq.match},
we have

$$
\Lcl(\beta,a) =  q^{\frac 1n \sum_{ij} \ell_{ij} |a_i| |a_j|}
1_\vartheta H(a) 1_\vartheta.
$$

Applying the evaluation map $\ev_n$ to both sides, using
Proposition \ref{r.Ladder} and the
normalization of $\tJ_L$ for the left hand side, we obtain that
$$
\tJ_L(e_{a_1},\dots,e_{a_k}) = \ev_n(H(a)).
$$
This proves \eqref{eq.m0}.

Let us have a closer look at the formula of $H$.
By \eqref{eq.tau+} and \eqref{eq.tau-}, $z_j$ has the form
\begin{align}
\label{eq.1f}
z_j &= \sum_{s_j \in \BZ} (-1)^{g_j(a,s_j)} q^{\ve_j h_j(a,s_j)}
E_{i_j}^{(f_j(a,s_j  ))} F_{i_j}^{(s_j)},
\end{align}
where $g_j$ is a quadratic function, and $h_j,f_j$ are linear functions.
From \eqref{eq.Xve}, it follows that $H$ has a presentation \eqref{eq.28},
where $s=(s_1,\dots,s_t)$, and
$$
Y^{(f(a,s))} = \Capam_m(a)
\left( \rprod_{j\in [1,t]} E_{i_j}^{(f_j(a,s_j  ))} F_{i_j}^{(s_j)} \right)
\Cupam_m(a).
$$

Assume $\|s \|_\infty > \|a \|_\infty$, i.e. there is $l$ such that
$|s_l| > \|a \|_\infty$. We can assume that $s_l >0$, since other wise
$s_l <0$ and the factor $F_{i_l}^{(s_l)}$ on the right hand side of
\eqref{eq.1f} is 0. We will show that the $(i_l+1)$-th component of the
weight of $F_{i_l}^{(s_l )} z$ is negative, where
$$
z =   \left( \rprod_{j\in [l+1,t]} E_{i_j}^{(f_j(a,s_j  ))}
F_{i_j}^{(s_j)} \right) \Cupam_m(a).
$$
This will prove that $(Y, f(a,s))$ is tail negative, since $i_l +1 >m$.
Note that
$$
\w (z) = \w (z_{l+1}(a) \dots z_t(a) \Cupam_m(a)) =
(c_1 -n,\dots, c_m -n, b'),
$$
where $b'$ is a permutation of $b$.
Since $\| b'\|_\infty  =\| b\|_\infty  = \| a\|_\infty$, we have
$$
\w_{i_l+1}(F_{i_l}^{(s_l )} z)=   - s_l + (b')_{i_l +1} < 0,
$$
which  completes the proof of the proposition.
\end{proof}

\begin{remark}
Our evaluation algorithm should be closely related to the variant of skew
Howe duality defined for so-called doubled Schur algebras in \cite{QS,QS2}.
\end{remark}

\subsection{An example: the trefoil knot}
\label{sub.trefoil}

Before we proceed further, let us illustrate Proposition \ref{r.main0}
by computing the invariant of the trefoil, and draw some useful
conclusions regarding $q$-holonomicity of the invariant.

We take
\be
\label{eq.E31}
\beta=\sigma_1^3 =\sigma_1 \sigma_1 \sigma_1, \qquad m=2,
\qquad \vartheta=(n,n,0,0) \,.
\ee
The link $L=\cl(\beta)$ is the right-handed trefoil knot, colored by one
$a\in \BN\cap [0,n-1]$, and $b=(a,a)$ (see Figure \ref{fig:trefoil}).

\begin{figure}[!hptb]
\begin{center}
\includegraphics[height=0.35\textheight]{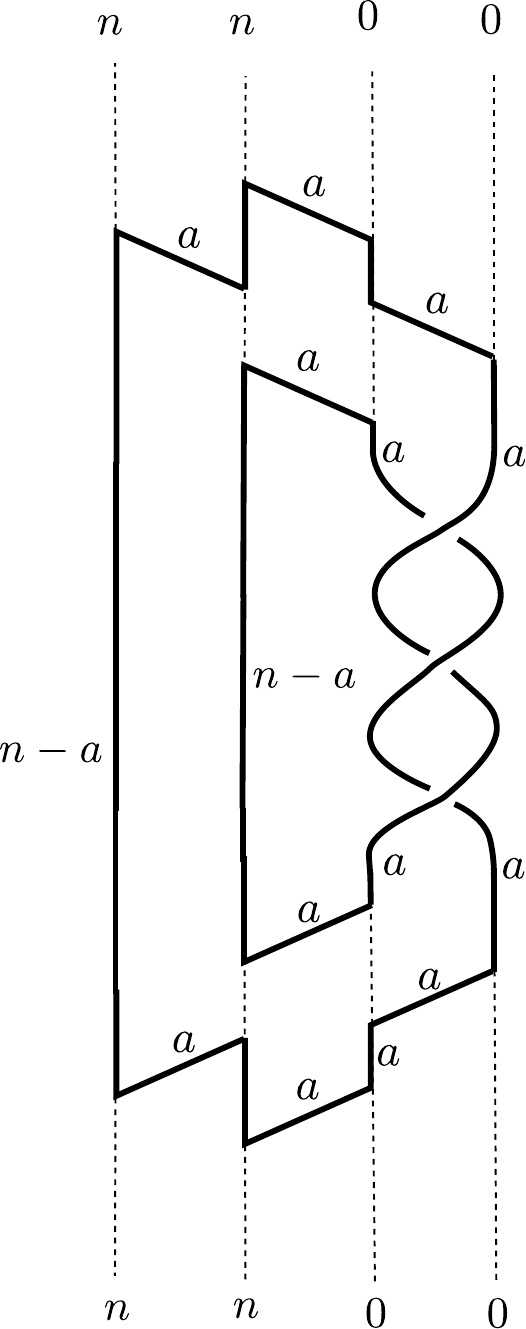}
\caption{The ladder closure of  braid $\beta=\sigma_1^3$.}
\label{fig:trefoil}
\end{center}
\end{figure}

By Proposition \ref{r.main0}, and equation~\eqref{eq.Xve} we obtain that
$\tJ_{3_1}^{\fsl_n}(e_a)=  \ev_n(H(a))$,
where
\be
\label{eq.Y31}
H(a)=
q^{3a} E_2^{(a)}E_1^{(a)}E_3^{(a)}E_2^{(a)}(T_3(b))^3
F_2^{(a)}F_3^{(a)}F_1^{(a)}F_2^{(a)}.
\ee
Using Equation \eqref{eq.tau+}, we replace each occurrence of $T_3(b)$
by a sum over the integers, and obtain the following triple sum formula
\begin{align}
\label{eq-exp1}
q^{-3a} H(a) =  \scs \xsum{s_1,s_2,s_3 \in \BZ}{}
(-q)^{s_1+s_2+s_3}E_2^{(a)}E_1^{(a)}E_3^{(a)}E_2^{(a)}
(E_3^{(s_1)}F_3^{(s_1)}E_3^{(s_2)}F_3^{(s_2)}E_3^{(s_3)}F_3^{(s_3)} )
F_2^{(a)}F_3^{(a)}F_1^{(a)}F_2^{(a)}.
\end{align}
This is an explicit form of special function for $H$.

Next, we use the commutation rules given in
equations~\eqref{eq.com1a}--\eqref{eq.com4a} to sort the  expression of
$H(a)1_{\vartheta}$, moving all the $E$'s to the right and all the $F$'s
to the left. Every time we move $\Eir$ (from the left) past an
$F_i^{(s)}$ (from the right), we obtain a 1-dimensional sum over
the integers. Then, use equation~\eqref{eq.Ei} to add some delta
functions in the sum. Finally, Equ.~\eqref{eq.8s}, which is explained
later in the proof of Proposition \ref{prop.tech1}, tells us to add
Heaviside functions $\He(k)$ (see Section \ref{sec.qholo}).
Doing so, we
eventually get the following formula for the quantum $\fsl_n$-invariant
of the trefoil colored by $e_a$ (the details are given in the appendix)

\begin{align}
\tJ_L^{\fsl_n}(e_a)
& =  \scs q^{3a} {n \brack a}
\sum_{s \in \BZ^6}  (-q)^{s_1+s_2+s_3} \He(a-s_1)\He(a-s_2) \He(a-s_3)
\He(a+s_1+s_2-s_4) \He(a+s_2+s_3-s_5) \He(\tau)\He(a-\tau) \nonumber \\
& \qquad
\scs {s_2+s_1 \brack s_4}
{s_2+s_3 \brack s_5}
{\tau+ s_2-s_6 \brack s_6} {s_1+s_2-s_4 \brack s_1}
{\tau \brack s_1+s_2-s_4}
{s_2+s_3-s_5 \brack s_3}
{\tau \brack s_2+s_3-s_5}
{n \brack a-\tau}
{n-\tau \brack a}
\label{eq.trefoil}
\end{align}
where $\tau=s_1+s_2+s_3-s_4-s_5-s_6$ and
$\und s=(s_1,\dots,s_6)\in \BZ^6$. Keep in mind the convention that
${ r \brack s} =0$ if $ s <0$.

Let us end this example with some observation. The above formula has the
form
\be
\label{eq.52}
\tJ_L^{\fsl_n}(e_a) = \sum_{s \in \BZ^6} F(a,s),
\ee
where $F(a,s)$ is a finite product of factors of the following shape
\begin{itemize}
\item[(i)]
$(\pm q)^{A(a,s)}$,
\item[(ii)]
$\He(A(a,s))$,
\item[(iii)]
quantum binomial ${A(a,s) \brack B(a,s)}$,
\item[(iv)]
quantum binomial ${n+ A(a,s) \brack B(a,s)} = {q^n; A(a,s) \brack B(a,s)}$,
where for $s,l \in \BZ$ we define
\be
\label{eq.qbn1}
{ x;  s \brack l} =\begin{cases} 0 \quad & \text{if } \ l < 0 \\
\prod_{j=1}^l
\frac{  x q^{ s -j+1}- x^{-1} q^{-s+j-1}  }{q^j - q^{-j}} &
\text{if } \ l \ge  0.
\end{cases}
\ee
\end{itemize}
Here $A(a,s)$ and $B(a,s)$ are $\BZ$-linear functions.
Moreover, for each {\em integer} value of $a$ and $n$, the sum on the
right hand side of \eqref{eq.52} is terminating in the sense that only
a finite number of terms are non-zero. The number of terms are bounded
by a polynomial function of $a$.

We will show that a similar formula exists for any framed oriented
link colored with $e_a$. But before we do so, let us recall what is a
$q$-holonomic function.


\section{$q$-holonomic functions}
\label{sec.qholo}

$q$-holonomic functions of one variable were introduced by the seminal paper
of Zeilberger \cite{Z90}. The class of $q$-holonomic functions resembles in
several ways and the class of holonomic $D$-modules, as is
acknowledged by conversations of Zeilberger and Bernstein
prior to the introduction of holonomic functions, \cite{Z90}. An extension
of the definition to $q$-holonomic functions with several variables was
given by Sabbah~\cite{Sabbah} using the language of homological algebra.
In this section we will review the definition of $q$-holonomic functions
of several variables, give examples, and list the closure properties of
this class under some operations. Our exposition follows Zeilberger
and Sabbah, and the survey paper of two of the authors~\cite{GL:survey}.

We should point out, however, that the precise definition of $q$-holonomic
functions is not used in the proof of Theorem~\ref{thm.1}. If the reader
wishes to take as a black box the examples of $q$-holonomic functions
given below, and their closure properties, then they can skip this section
altogether and still deduce the proof of Theorem~\ref{thm.1}.

\subsection{The quantum Weyl algebra}
\label{sub.weyl}

Let $V$ denote a fixed (not necessarily finitely generated) $\cA$-module,
where $\cA= \BZ[q^{\pm1}]$.
For a natural number $r$, let $S_r(V)$ be the set of all functions
$f: \BZ^r \to V$ and $S_{r,+}(V)$ the set of functions $f: \BN^r \to V$.
For $i=1,\dots, r$ consider the operators $\sL_i$ and $\sM_i$ which act
on functions $f \in S_r(V)$ by
\begin{align}
(\sL_if)(n_1,\dots, n_i, \dots, n_r) &= f(n_1, \dots, n_i+1, \dots , n_r)\\
(\sM_if)(n_1, \dots, n_r)&= q^{n_i} f(n_1, \dots, n_r).
\end{align}
It is clear that $\sL_i$, $\sM_j$ are invertible operators that satisfy
the $q$-commutation relations
\begin{subequations}
\begin{align}
\sM_i \sM_j &= \sM_j \sM_i
\label{eq.w1}\\
\sL_i \sL_j &=
\sL_j \sL_i
\label{eq.w2}\\
\sL_i \sM_j &= q^{\delta_{i,j}} \sM_j \sL_i
\label{eq.w3}
\end{align}
\end{subequations}
for all $i,j=1,\dots,r$.

\begin{definition}
\label{def.weyl}
The $r$-dimensional quantum Weyl algebra $\cW_r$ is the
$\cA$-algebra generated by
$\sL_1^{\pm 1}, \dots, \sL_r^{\pm 1}, \sM_1^{\pm 1}, \dots, \sM_r^{\pm 1}$
subject to the relations \eqref{eq.w1}--\eqref{eq.w3}.
Let $\cW_{r,+}$ be the subalgebra of $\cW_r$ generated by the non-negative
powers of $\sM_j, \sL_j$.
\end{definition}

Given $f \in S_r(V)$,
the annihilator $\Ann(f)$ (a left $\cW_r$ module) is defined by
\be
\label{eq.ann}
\Ann(f)=\{P \in \cW_r \,| Pf=0\}
\ee
This gives rise to a cyclic $\cW_r$-module $M_f$, defined by
$M_f = \cW_r f \subset S_r(V)$, and isomorphic to $\cW_r/\Ann(f)$.

\subsection{Definition of $q$-holonomic functions}
\label{sub.definition}

In this section we follow closely the work of Sabbah~\cite{Sabbah}.
Let $N$ be a finitely-generated $\cW_{r,+}$-module. Consider
the increasing filtration $\calF$ on $\cT_{r,+}$ given
\be
\label{eq.FN}
\calF_n \cT_{r,+} = \{\text{$\cA$-span of all monomials
$\sM^\alpha \sL^\beta$ with $\alpha,\beta \in \BN^r$
with total degree $\le n$}\} \,.
\ee
The filtration $\calF$ on $\cW_{r,+}$
induces an increasing filtration on $N$, defined by
$\calF_n N=\calF_n \cW_{r,+} \cdot N$. Note that $\calF_n \cW_{r,+}$,
and consequently $\calF_n N$, are  finitely-generated $\cA$-modules
for all natural numbers $n$.
An analog of Hilbert's theorem for this non-commutative setting holds:
the dimension of the $\BQ(q)$-vector space $ \BQ(q) \otimes_{\cA}  \calF_n N$
is a polynomial in $n$, for big enough $n$. The degree of this polynomial
is call   the dimension of $N$, and is denoted by $d(N)$.

In \cite[Theorem 1.5.3]{Sabbah} Sabbah proved that $d(N) = 2r -\codim(N)$,
where
$$
\codim(N) = \min\{ j\in \BN \mid \Ext^j_{\cW_{r,+}}(N, \cW_{r,+}) \neq 0 \}.
$$
Sabbah also proved that $d(N) \ge r$ if $N$ is a non-zero and does not have
monomial torsion. Here a monomial torsion is a monomial  $P$ in  $\cW_{r,+}$
such that $Px=0$ for a certain non-zero $x\in N$.
It is easy to see that $N$ embeds in the $\cW_r$-module
$\cW_r\otimes_{\cW_{r,+}} N$ if and only if $N$ has no monomial torsion.
Throughout the paper, we will assume that all $\cW_{r,+}$-modules do not
have monomial torsion.

\begin{definition}
\rm{(a)}
A $\cW_{r,+}$-module $N$ is $q$-holonomic if $N=0$ or $N$ is
finitely-generated, does not have monomial torsion, and $d(N)=r$.
\newline
\rm{(b)}
An element $f\in N$, where $N$ is a $\cW_{r,+}$-module, is $q$-holonomic
over $\cW_{r,+}$ if $\cW_{r,+}\cdot f$ is a $q$-holonomic $\cW_{r,+}$-module.
\end{definition}

The above definition defines $q$-holonomic $\cW_{r,+}$ modules, and our
next task is to define $q$-holonomic $\cW_r$ modules.
Let $M$ be a non-zero finitely-generated left $\cT_r$-module.
Following \cite[Section 2.1]{Sabbah}, the codimension and dimension of $M$
are defined in terms of homological algebra by
$$
\codim(M) = \min\{ j\in \BN \mid \Ext^j_{\cW_r}(M, \cT_r) \neq 0 \},
\quad \dim(M) = 2r - \codim(M).
$$
\noindent
The key Bernstein inequality (proved by Sabbah \cite[Thm.2.1.1]{Sabbah}
in the $q$-case) asserts that if $M\neq 0$ is a finitely generated
$\cW_r$-module, then $\dim(M) \ge r$.

\begin{definition}
\rm{(a)}
A $\cW_r$-module $M$ is $q$-holonomic if either $M=0$ or $M$
is finitely-generated and $\dim(M)= r$.
\newline
\rm{(b)}
An element $f\in M$, where $M$ is a $\cW_r$-module, is
$q$-holonomic over $\cW_r$ if $\cW_r\cdot f$ is a $q$-holonomic $\cW_r$-module.
\end{definition}

Next we compare $q$-holonomic modules over $\cW_r$ versus over $\cW_{r,+}$.
The following proposition was proven in \cite[Sec.3]{GL:survey}
Next we compare $q$-holonomic modules over $\cW_r$ versus over $\cW_{r,+}$.
using Sabbah~\cite[Cor.2.1.4]{Sabbah}.

\begin{proposition}
\label{prop.WWr}
Suppose $f\in M$, where $M$ is a $\cW_r$-module. Then $f$ is $q$-holonomic
over $\cW_r$ if and only if it is $q$-holonomic over $\cW_{r,+}$.
\end{proposition}

The next corollary is taken from \cite[Sec.3]{GL:survey}.

\begin{corollary}
\label{cor.Z2N}
If $f \in S_r(V)$ is $q$-holonomic and $g \in S_{r,+}(V)$
is its restriction to $\BN^r$, then $g$ is $q$-holonomic.
\end{corollary}

\begin{remark}
\label{rem.AR}
The definition of $q$-holonomic $\cA$-modules can be extended to
$q$-holonomic $\cR$-modules where $\cR$ is the ring
(and also an $\cA$-module)
\be
\label{eq.cR}
\cR=\BQ(q)[x^{\pm 1}] \,.
\ee
Proposition \ref{prop.WWr} and Theorems \ref{thm.prop} and
\ref{thm.multisum} below hold after replacing $\cA$ by $\cR$.
\end{remark}

\subsection{Properties of $q$-holonomic functions}
\label{sub.properties}

In this section we summarize some closure properties of $q$-holonomic
functions, whose proofs can be found in \cite[Sec.5]{GL:survey}.

\begin{theorem}
\label{thm.prop}
Suppose $f,g\in S_r(V)$ are $q$-holonomic functions. Then,
\newline
\rm{(a)}  $f+g$ is $q$-holonomic.
\newline
\rm{(b)} $fg$ is $q$-holonomic.
\newline
\rm{(c)} Restriction.
For a fixed $a\in \BZ$, the function $g\in S_{r-1}(V)$ defined by
$$
g(n_1,\dots,n_{r-1})= f(n_1,\dots,n_{r-1},a)
$$
is $q$-holonomic.
\newline
\rm{(d)} Extension. The function $h \in S_{r+1}(V)$ defined by
$$
h(n_1,\dots,n_{r+1})= f(n_1,\dots,n_{r})
$$
is $q$-holonomic.
\newline
\rm{(e)} Linear substitution. If $A \in \GL(r,\BZ)$ and $f \in S_r(V)$
is $q$-holonomic, so is the composition $A \circ f \in S_r(V)$.
\end{theorem}

Let $S_{r-1,1}(V)$ denote set of all functions
$f: \BZ^r \to V$ such that for every $(n_1,\dots,n_{r-1})\in \BZ^{r-1}$,
$f(n_1,\dots,n_r)=0$ for all but a finite number of $n_r$.

\begin{theorem}
\label{thm.multisum}
\rm{(a)} Suppose $f\in S_{r-1,1}(V)$ is $q$-holonomic. Then,
$g\in S_{r-1}(V)$, defined by
$$
g(n_1,\dots,n_{r-1})= \sum_{n_r \in \BZ} f(n_1,\dots,n_r) \,,
$$
is $q$-holonomic.
\newline
\rm{(b)} Suppose $f\in S_r(V)$ is $q$-holonomic. Then
$h\in S_{r+1}(V)$ defined by
\be
\label{eq.multi}
h(n_1,\dots,n_{r-1},a,b)=\sum_{n_r=a}^b f(n_1,n_2,\dots,n_r)
\ee
is $q$-holonomic.
\end{theorem}

\subsection{Elementary $q$-holonomic functions}
\label{sub.elementary}

A function $g:\ \BZ^s \to \BZ^r$ is {\em affine} if there is an
$r\times s$ matrix $A$ with integer entries and $b \in \BZ^r$ such
that $g(a) = A a + b$. If $b=0$, such a function is called
{\em linear}.

A function $f: \BZ^r \to \Zxq$ is called an {\em elementary block}
if $f$ is a finite product of a composition of a linear function
$\BZ^r \to \BZ^s$ (for $s=1,2$) with one of one of the following functions:
\begin{itemize}
\item[(i)]
$\BZ \to \Zq, \quad k \to (-1) ^k$, or $k \to q^k$, or $k \to \He(k)$,
\item[(iii)]
$\BZ^2 \to \Zq, \quad (k,l) \to \delta_{k,l}$, or $(k,l) \to \qbinom k l$,
or $(k,l) \to \qbinom {x;k} l$.
\end{itemize}
Observe that functions of the form (i) or (ii) above are
$q$-holonomic~\cite{GL:survey}.

A function $f: \BZ^r \to \Zxq$ is called {\em elementary} if
can be presented by a terminating sum
$$
f(a)= \sum_{b  \in \BZ^l} g(a,b),
$$
where $g: \BZ^{k+l}\to \Zxq$ is an elementary block. Here the sum is
terminating means for each $a$ there are only a finite number of $b$
such that $g(a,b)\neq 0$. Theorems \ref{thm.prop} and \ref{thm.multisum}
imply the following.

\begin{corollary}
\label{cor.elementary}
Every elementary block and every elementary function is $q$-holonomic.
\end{corollary}


\section{Proof of Theorem \ref{thm.1}}
\label{sec.thm1}

\subsection{Evaluation of monomials is $q$-holonomic}\label{sec:evalqh}

For $n\in \BZ$ let $\EV_n: \Zxq \to \BQ(q)$ be the $\Qq$-algebra
homomorphism defined by
\be
\label{eq.EV}
\EV_n(f) = f|_{x= q^n}.
\ee

The next lemma recovers an element of $\Zxq$ from its evaluations.

\begin{lemma}
\label{r.uni1}
Suppose that $f,g \in \Zxq$ satisfy
$\EV_n(f)= \EV_n(g)$ for infinitely many $n$. Then $f=g$.
\end{lemma}

\begin{proof}
This follows from the fact that a Laurent polynomial in $x$ has at most
$k$ roots, where $k$ is the difference between the highest order and
the lowest order in $x$.
\end{proof}

Let $X=(X_1,\dots,X_k)$ be a sequence of elements of the set
$\calX_{2m}$ from equation \eqref{eq.calX}. Recall that for
$b=(b_1,\dots,b_k) \in \BZ^k$,  the monomial $X^{(b)}\in \Ass_{2m}$ and
its weight are defined in Section \ref{sub.free}. By convention,
$X^{(b)}=0$ if one of $b_i$ is negative. The goal of this subsection is
to calculate
$$
\ev_n(X^{(b)}) = \ev ( 1_\vartheta X^{(b)} 1_\vartheta)
$$
where $\vartheta = (n^m,0^m) \in \BZ^{2m}$.

\begin{proposition}
\label{prop.tech1}
Suppose $X=(X_1,\dots,X_k)$ is a sequence of elements of the set $\calX_{2m}$.
There exists a unique function
$$
Q_X:  \BZ^k \to \BQ(q)[x^{\pm 1}]
$$
such that for all $b \in \BZ^k$,
\be
\label{eq.a3a}
\ev_n(X^{(b)}) = \EV_n(Q_X(b)).
\ee
Moreover, $Q_X$ is an elementary function given by
\be
\label{eq.sum1}
Q_X(b) = \sum_{j\in \BZ^l} F_X(b,j)
\ee
for certain $ l \in \BN$ and elementary block $F_X: \BZ^{k+l} \to \Zxq$.
In addition,
\begin{itemize}
\item[(i)]
$F_X(b,j)=0$ if $\|j \|_\infty > \|b \|_\infty$ (which implies the
sum \eqref{eq.sum1} is terminating), and
\item[(ii)]
$F_X(b,j)=0$ if $(X,b)$ is tail-negative if one of the component
of $b$ is negative.
\end{itemize}
\end{proposition}

\begin{proof}
The uniqueness follows from Lemma \ref{r.uni1}. Let us prove the existence.

The idea is to move the $E_i$ to the right of $F_j$ using Equations
\eqref{eq.com1a} and \eqref{eq.com2a} (this creates a sum of a product
of $q$-binomials) and then use Equation \eqref{eq.Ei},
which creates a product of $\delta$-functions. Besides, we insert Heaviside
functions to make the sum terminating. The result is an elementary
function. Now we give the details of the proof.

Let $l \leq k$ be the maximal index such that $X_l\in\{ E_1,\dots, E_{2m-1}\}$.
We use induction on $k$, then induction on $l$.
If $k=0$ the statement is obvious.

For fixed $k$, we use induction on $l$, beginning with $l=k$ and going down.

(a) Suppose $l=k$. Recall that $\vartheta=(n^m, 0^m)$.
Using \eqref{eq.Ei}, we have
$$
X^{(b)} 1_{\vartheta}= \delta_{b_k,0} Y^{(b')} 1_{\vartheta},
$$
where $Y=(X_1,\dots, X_{k-1})$ and $b'=(b_1,\dots,b_{k-1})$. For $Y$
the statement holds. Define
$$
F_X(b,j):= F_Y(b',j) \delta_{b_k,0}, \quad
Q_X(b)= \sum_{j \in \BZ^l} F_X(b,j)
$$
Then $F_X(b,j)$ is an elementary summand. Both statements (i) and
(ii) for $F_X(b,j)$ follow immediately from those for $F_Y(b',j)$.
Then $Q_X$ is an elementary $q$-holonomic function, and \eqref{eq.a3a} holds.

(b) Suppose $l<k$. Assume that $X_l=E_r$ and $X_{l+1}=F_s$. Let
$Y=(Y_1,\dots,Y_k)$ be the sequence defined by $Y_i=X_i$ for all $i$
except $Y_{l}= X_{l+1}$ and $Y_{l+1}= X_l$. By induction, the statement holds
for $Y$, and we can define elementary summand $F_Y(b,j)$ for
$(b,j) \in \BZ^{k+l}$. Consider two cases.

Case 1: $r\neq s$. Because $E_r F_s = F_s E_r$, we have
$X^{(b)} = Y^{(b')}$ where $b'$ is obtained from $b$ by swapping
the $l$-th and $(l+1)$-components. This case is reduced to the case of $Y$
by defining $F_X(b,j)= F_Y(b', j)$.

Case 2: $r=s$. We have
\begin{align*}
X^{(b)} &
=  X_{\lef}\, \big( E_r^{(b_l)} F_r^{(b_{l+1})}\big) \, X_{\rig}
\end{align*}
where
$$
X_\lef=  \rprod_{j\in [1,l-1]} X_j^{(b_j)}, \quad X_\rig =
\rprod_{j\in [l+2,k]} X_j^{(b_j)}.$$
We have  $X_\rig 1_\vartheta \in \Hom_{n\Lad_{2m}}(\mu, \vartheta)$, where
$$
\mu = \vartheta + \w(X_\rig) = \vartheta -\sum_{j=l+2}^k b_j \alpha_{i_j}.
$$
Here the index $i_j$ is defined so that $X_j= F_{i_j}$ for $j>l$.
Using  \eqref{eq.com1}, we have
\begin{align}
X^{(b)} 1_\vartheta
&= \sum_{t\in \BZ}\qbinom{ \langle \mu, \al_r \rangle
+ b_l - b_{l+1}}{t}  X_{\lef} F_r^{(b_{l+1}-t)} E_r^{(b_l-t)}  X_{\rig} 1_\vartheta
\nonumber \\
&= \sum_{t\in \BZ}\qbinom{ \langle \mu, \al_r \rangle
+ b_l - b_{l+1}}{t}  Y^{(b')} 1_\vartheta,
\label{eq.a3}
\end{align}
where $b'=(b'_1,\dots,b'_k)$ such that $b'_i= b_i$ for all $i$ except
for $i=l,l+1$,  with $b'_{l}= b_{l+1}-t, b'_{l+1}=b_l-t$. Clearly $b'$ is
a linear function of $(b,t)$.

Note that $\langle \vartheta, \al_r\ra= n \delta(r,m)$. From the definition
of $\mu$,
\be
\nonumber
\langle \mu, \al_r \rangle + b_l - b_{l+1}= n \delta(r,m) + \Lin(b)
\ee
where $\Lin(b) = \langle \w(X_\rig), \al_r\ra + b_l - b_{l+1}$ is a
$\BZ$-linear form of $b$. For $j\in \BZ^{l+1}$ we write $j=(j', t)$,
i.e. $t$ is the last component of $j$. For $b\in \BZ^k$ and
$j\in \BZ^{l+1}$, define

\be
\label{eq.8s}
F_X(b,j) = \begin{cases}\qbinom{ x; \Lin(b)}{t} F_Y(b', j')
\, \sH(X,b)  \quad &\text{if } \ r=m
\\
\qbinom{ \Lin(b)}{t} F_Y(b', j')\,  \sH(X,b)   &
\text{ if } \ r \neq m, \end{cases}
\ee
where $\sH(X,b)$, defined by \eqref{eq.HH}, is an elementary
function of $b$.

Then $ F_X(b,j)$ is an elementary function. Let us prove (i) and (ii),
which claim $F_X(b,j)=0$ under certain conditions. If $t<0$, then the
first factor on the right hand sides of \eqref{eq.8s} is 0. Hence we will
assume $t\ge0$ in what follows.

(i) Suppose $\|j\|_\infty > \|b\|_\infty$. Then either
$\|j'\|_\infty > \|b\|_\infty $ or $|t| > \|b\|_\infty$.  In the first
case, $\|j'\|_\infty > \|b\|_\infty \ge  \|b'\|_\infty$, and
 $F_Y(b', j')=0$.
In the latter case, the $l$-th component of $b'$ is negative. By (ii) we
have $F_Y(b', j')=0$. Hence $F_Y(b, j)=0$.

(ii) First assume that one of the component of $b$ is negative.
 Then  one of the component of $b'$ is negative. Hence
$F_Y(b', j')=0$, implying $F_X(b, j)=0$.

Now assume that $(X,b)$ is tail-negative. Then the third factor in the
right hand sides of \eqref{eq.8s} is 0. Hence $F_X(b, j)=0$.

Let us prove \eqref{eq.a3a}. If $(X,b)$ is tail-negative, then
both sides of \eqref{eq.a3a} are 0, by Lemma \ref{r.00} and the property of
$\sH(X,b)$.
Assume now $(X,b)$ is not tail-negative. Then $\sH(X,b)=1$, and
\eqref{eq.a3a} follows from \eqref{eq.a3}, \eqref{eq.8s}, and the identity
\eqref{eq.a3a} applicable to $Y$.

This completes the proof of the proposition.
\end{proof}

\subsection{Coloring with partitions with one column}

\begin{theorem}
\label{thm.2}
Suppose $L$ is an oriented, framed link with $r$ ordered components.
There exists a unique function
$$
Q_L: \BN^r \to \Zxq
$$
such that for any integer $n \ge 2$ and
$a=(a_1,\dots,a_r) \in \BN^r \cap [0,n-1]^r$,
\be
\label{eq.holo1}
\tJ_L^{\fsl_n}(e_{a_1},  \dots, e_{a_r}) = \EV_n\left( Q_L(a)\right).
\ee
Moreover, $Q_L$ is elementary, hence a $q$-holonomic function.
\end{theorem}

\begin{proof}
The uniqueness follows from Lemma \ref{r.uni1}. Let us prove the existence.
Suppose $L$ is the closure of a braid $\beta$ on $m$ strands, as in
Section \ref{sec.Homf}. By Proposition \ref{r.main0}, there exists a
sequence $X=(X_1,\dots X_k)$ of elements from $E_i,F_i$ with
$i=1,\dots, 2m-1$ and linear functions $g_1, g_2:\BZ^{r+t} \to \BZ$ and
$f: \BZ^{r+t} \to \BZ^k$ such that
$$
\tJ_L^{\fsl_n}(e_{a_1}, , \dots, e_{a_r})) =
\sum_{s \in \BZ^t } (-1)^{g_1(a,s)} q^{g_2(a,s)} \ev_n
\left( X^{(f(a,s))} \right).
$$
By Proposition \eqref{prop.tech1}, there exists elementary summand
function $F_X: \BZ^{k+l}\to \Zxq$ such that
$$
\tJ_L^{\fsl_n}(e_{a_1}, , \dots, e_{a_r}) =
\sum_{s \in \BZ^t } (-1)^{g_1(a,s)} q^{g_2(a,s)} \eval_n
\left( \sum_{j \in \BZ^l} F_X(f(a,s),j)\right).
$$
By (i) of Proposition \ref{prop.tech1},
\begin{align}
\label{eq.5a}
F_X (f(a,s),j)& = 0 \quad \text{if   }\  \|j \|_\infty >
\|f(a,s) \|_\infty.
\end{align}

When  $\|s \|_\infty >  \|a \|_\infty$, $(X,f(a,s))$ is
tail-negative, see Proposition \ref{r.main0} . Hence, by (ii) of
Proposition \ref{prop.tech1},

\begin{align}\label{eq.5b}
F_X(f(a,s),j)& = 0 \quad \text{if   }\  \|s \|_\infty >
\|a \|_\infty.
\end{align}
Equations \eqref{eq.5a} and \eqref{eq.5b} imply that the sum
$$
Q_L(a):= \sum_{s \in \BZ^t} \sum_{ j \in \BZ^r}
(-1)^{g_1(a,s)} q^{g_2(a,s)} F(f(a,s),j)
$$
is terminating for each $a \in \BZ^k$.

Then $Q_L$ is elementary $q$-holonomic, and equation \eqref{eq.holo1} holds.
\end{proof}

\begin{remark}
\label{rem.NZQ}
By our construction, $Q_L$ vanishes in $\BZ^r \setminus \BN^r$.
\end{remark}

\begin{remark}
\label{rem.2alt}
Theorem \ref{thm.2} gives an alternative construction of the
colored \HOMFLY\ polynomial $W_L$ of a framed, oriented link colored by
partitions with one column. By the uniqueness,
$$
Q_L(a_1,\dots,a_r)= W_L(e_{a_1}, \dots, e_{a_r}).
$$
\end{remark}


\subsection{The Jacobi-Trudi formula}
\label{sub.jacobi-trudy}

In this section we explain how to extend the $q$-holonomicity of the
\HOMFLY\ polynomial of a link colored by partitions with one row to
the case of partitions with a fixed number of rows. The key idea is
the Jacobi-Trudi formula which expresses the Schur function
$s_\lambda$ of a partition $\lambda\in \cP_\ell$,
considered as an element of the algebra $\Lambda$, as a determinant of a
matrix whose entries are  partitions with one row. Observe that
for partitions with one row (resp. one column) we have $s_{(a)}=h_a$
(resp., $s_{(1^a)}=e_a$).

The Jacobi-Trudi formula \cite{Mc} states that if
$\lambda=(\lambda_1,\dots,\lambda_\ell)\in \cP_\ell$, then in $\Lambda$,
$$
s_{\lambda}=\det\left( (e_{\lambda^\dagger_i+j-i})_{i,j=1}^\ell\right)
$$
where the right hand side is an $\ell \times \ell$ determinant, with the
convention $e_0=1$ and $e_n=0$ for $n<0$.
For example, if $\lambda$ is a partition with three rows with $\lambda_1$,
$\lambda_2$ and $\lambda_3$ boxes, then we have
\begin{align*}
s_{\lambda_1,\lambda_2,\lambda_3} =&
-e_{\lambda_{1}+2} e_{\lambda_{2}} e_{\lambda_{3}-2}
+e_{\lambda_{1}+1} e_{\lambda_{2}+1} e_{\lambda_{3}-2}
+e_{\lambda_{1}+2} e_{\lambda_{2}-1} e_{\lambda_{3}-1} \\
& -e_{\lambda_{1}} e_{\lambda_{2}+1} e_{\lambda_{3}-1}
-e_{\lambda_{1}+1} e_{\lambda_{2}-1} e_{\lambda_{3}}+h
_{\lambda_{1}} e_{\lambda_{2}} e_{\lambda_{3}}.
\end{align*}

Let $L$ denote a framed, oriented link $L$ with $r$ ordered components,
and choose a partition $\lambda\in \cP_\ell$, and partitions
$\mu_2, \dots, \mu_r$. Then, part (c) of Proposition \ref{r.WL} implies that
\be
\label{eq.JT}
W_L(\lambda,\mu_1,\dots,\mu_r)= \sum_{\sigma \in \mathrm{Sym}_\ell}
\mathrm{sgn}(\sigma) W_{L'} (e_{\lambda_1+\sigma(1)-1},\dots,
e_{\lambda_\ell+\sigma(\ell)-\ell},\mu_1,\dots,\mu_r).
\ee
where $L'$ is the link obtain from $L$ by replacing the first framed
component of $L$ by  $\ell$ of its parallels.

\subsection{Proof of Theorem \ref{thm.1}}
\label{sub.end.thm1}

Fix a framed oriented link $L$ with $r$ ordered components.
Using the symmetry of the \HOMFLY\ polynomial from part (c) of
Proposition \ref{r.WL}, it suffices to show that the colored \HOMFLY\
polynomial of $L$, colored by partitions with at most $\ell$ rows, is
$q$-holonomic. Said differently, it suffices to show that the function
$W_L \circ (\iota^\dagger_\ell)^r : \BN^{r \ell} \to \Zxq$ is $q$-holonomic.
Let $\lambda=(\lambda_1,\dots,\lambda_{r \ell}) \in \BN^{r \ell}$.
Using Equation \eqref{eq.JT}, we have
$$
(W_L \circ (\iota^\dagger_\ell)^r)(\lambda) = \sum_{\sigma}
\mathrm{sgn}(\sigma) W_{\Delta L}(
e_{f_{\sigma,1}(\lambda)},\dots, e_{f_{\sigma,r \ell}(\lambda)})
$$
where the sum is over
$\sigma=(\sigma_1,\dots,\sigma_r) \in (\mathrm{Sym}_\ell)^r$,
$\mathrm{sgn}(\sigma)=\mathrm{sgn}(\sigma_1) \dots \mathrm{sgn}(\sigma_r)$
and $\Delta L$ is the link with $r \ell$ components
obtained from $L$ by replacing each component with its
$\ell$-th parallel and $f_{\sigma,i}: \BZ^{r \ell} \to \BZ$ are affine.
Parts (a) and (e) of Theorem \ref{thm.prop} together with Theorem
\ref{thm.2} imply that $W_L \circ (\iota^\dagger_\ell)^r)$ is a sum of
$q$-holonomic functions, thus is $q$-holonomic.
This concludes the proof of Theorem~\ref{thm.1}.
\qed

\subsection*{Acknowledgments}
S.G. and T.L. were supported in part by the National Science
Foundation grant DMS-14-06419. A.D.L. was partially supported by NSF grant
DMS-1255334, the John Templeton Foundation, and the Swiss MAP NCCR grant of
the Swiss National Science Foundation. We wish to thank Christian Blanchet,
Kenichi Kawagoe, Hugh Morton for helpful discussions, and Scott Morrison
for numerous conversations.


\appendix
\section{The formula for the invariant of the trefoil}
\label{sec.app1}

In this section we give the omitted details of how equation~\eqref{eq-exp1}
implies equation~\eqref{eq.trefoil}. We start with equation~\eqref{eq-exp1},
and observe that
\begin{align} \label{eq.tref}
 &  \scs \xsum{s_1,s_2,s_3 \in \BZ}{}
(-q)^{s_1+s_2+s_3}E_2^{(a)}E_1^{(a)}E_3^{(a)}E_2^{(a)}
E_3^{(s_1)}F_3^{(s_1)}E_3^{(s_2)}F_3^{(s_2)}E_3^{(s_3)}F_3^{(s_3)}
F_2^{(a)}F_3^{(a)}F_1^{(a)}F_2^{(a)}1_{(n,n,0,0)}
\\ \nonumber
  &=  \scs \xsum{s_1,s_2,s_3 \in \BZ}{}
(-q)^{s_1+s_2+s_3}E_2^{(a)}E_1^{(a)}E_3^{(a)}E_2^{(a)}
\big(E_3^{(s_1)}F_3^{(s_1)}E_3^{(s_2)}F_3^{(s_2)}E_3^{(s_3)}F_3^{(s_3)}
1_{(n-a,n-a,a,a)} \big)
F_2^{(a)}F_3^{(a)}F_1^{(a)}F_2^{(a)}1_{(n,n,0,0)}
\end{align}
where we used \eqref{eq.com0}  to include the idempotent in the middle term
(and the fact that $(n,n,0,0)-a\alpha_1-2a\alpha_2-a\alpha_3 = (n-a,n-a,a,a)$.
The term in parenthesis can be simplified as follows.
\begingroup
\allowdisplaybreaks
\begin{align}
& \scs E_3^{(s_1)}F_3^{(s_1)}E_3^{(s_2)}F_3^{(s_2)}\big(E_3^{(s_3)}F_3^{(s_3)}
1_{(n-a,n-a,a,a)} \big)
\nonumber \\
&\scs  \refequal{\eqref{eq.com1}}
E_3^{(s_1)}F_3^{(s_1)}F_3^{(s_2)}E_3^{(s_2)}F_3^{(s_3)}E_3^{(s_3)}
1_{(n-a,n-a,a,a)}
\nonumber \\
&\scs  \refequal{\eqref{eq.com0}}
\big(F_3^{(s_1)}E_3^{(s_1)}1_{(n-a,n-a,a,a)}\big)
F_3^{(s_2)}E_3^{(s_2)}F_3^{(s_3)}E_3^{(s_3)}
\nonumber \\
&\scs  \refequal{\eqref{eq.com1}}
F_3^{(s_1)}E_3^{(s_1)}F_3^{(s_2)}E_3^{(s_2)}F_3^{(s_3)}E_3^{(s_3)}
1_{(n-a,n-a,a,a)}
\nonumber \\
&\scs   \refequal{\eqref{eq.com0}}
F_3^{(s_1)} 1_{(n-a,n-a,a+s_1,a-s_1)}E_3^{(s_1)}F_3^{(s_2)} 1_{(n-a,n-a,a+s_2,a-s_2)}
E_3^{(s_2)} F_3^{(s_3)}1_{(n-a,n-a,a+s_3,a-s_3)} E_3^{(s_3)}
 \nonumber \\
 &\scs \refequal{\eqref{r.00}}
 \He(a-s_1)\He(a-s_2) \He(a-s_3)F_3^{(s_1)}\big(E_3^{(s_1)}F_3^{(s_2)}
1_{(n-a,n-a,a+s_2,a-s_2)}\big)
 \big(E_3^{(s_2)} F_3^{(s_3)}1_{(n-a,n-a,a+s_3,a-s_3)} \big) E_3^{(s_3)}
 \nonumber \\
&\scs  \refequal{\eqref{eq.com1}}
\He(a-s_1)\He(a-s_2) \He(a-s_3)\xsum{s_4,s_5}{}
{s_2+s_1 \brack s_4}{s_2+s_3 \brack s_5}
 F_3^{(s_1)} \big(F_3^{(s_2-s_4)}
E_3^{(s_1-s_4)}1_{(n-a,n-a,a+s_2,a-s_2)}\big) \nonumber \\
& \qquad \qquad \qquad \qquad \qquad \scs \times
 \big(F_3^{(s_3-s_5)} E_3^{(s_2-s_5)}1_{(n-a,n-a,a+s_3,a-s_3)}\big)
 E_3^{(s_3)}  1_{(n-a,n-a,a,a)}  \nonumber \\
&\scs  \refequal{\eqref{eq.com0}}
 \He(a-s_1)\He(a-s_2) \He(a-s_3)
  \xsum{s_4,s_5}{} {s_2+s_1 \brack s_4}{s_2+s_3 \brack s_5}
 F_3^{(s_1)} F_3^{(s_2-s_4)}  1_{(n-a,n-a,a+s_1+s_2-s_4,a-s_1-s_2+s_4)}
 \nonumber \\
& \qquad \qquad \qquad \scs \times
 \big(E_3^{(s_1-s_4)} F_3^{(s_3-s_5)}  1_{(n-a,n-a,a+s_2+s_3-s_5,a-s_2-s_3+s_5)} \big)
 E_3^{(s_2-s_5)} E_3^{(s_3)}  1_{(n-a,n-a,a,a)}  \nonumber \\
&\scs  \refequal{\eqref{r.00}}
  \He(a-s_1)\He(a-s_2) \He(a-s_3) \He(a+s_1+s_2-s_4) \He(a+s_2+s_3-s_5)\times
 \nonumber \\
 & \quad  \scs \xsum{s_4,s_5}{} {s_2+s_1 \brack s_4}{s_2+s_3 \brack s_5}
 F_3^{(s_1)} F_3^{(s_2-s_4)}
 \big(E_3^{(s_1-s_4)} F_3^{(s_3-s_5)}  1_{(n-a,n-a,a+s_3+s_2-s_5,a-s_3-s_2+s_5)} \big)
 E_3^{(s_2-s_5)} E_3^{(s_3)}  1_{(n-a,n-a,a,a)}  \nonumber \\
&\scs  \refequal{\eqref{eq.com1}}
\scs \He(a-s_1)\He(a-s_2) \He(a-s_3) \He(a+s_1+s_2-s_4) \He(a+s_2+s_3-s_5)\times
 \nonumber \\
 & \quad  \scs
\xsum{s_4,s_5,s_6}{}
 {s_2+s_1 \brack s_4}
 {s_2+s_3 \brack s_5}
 {s_1+2s_2+s_3-s_4-s_5 \brack s_6}
 F_3^{(s_1)} F_3^{(s_2-s_4)}\big(F_3^{(s_3-s_5-s_6)}
 E_3^{(s_1-s_4-s_6)}\big)E_3^{(s_2-s_5)}
E_3^{(s_3)}  1_{(n-a,n-a,a,a)} \nonumber \\
&\scs  \refequal{\eqref{r.00}}
\scs \He(a-s_1)\He(a-s_2) \He(a-s_3) \He(a+s_1+s_2-s_4) \He(a+s_2+s_3-s_5)
\He(a+s_1+s_2+s_3-s_4-s_5-s_6)\times
 \nonumber \\
 & \quad  \scs
\xsum{s_4,s_5,s_6}{}
 {s_2+s_1 \brack s_4}
 {s_2+s_3 \brack s_5}
 {s_1+2s_2+s_3-s_4-s_5 \brack s_6}
\big( F_3^{(s_1)} F_3^{(s_2-s_4)}F_3^{(s_3-s_5-s_6)}\big)
 \big(E_3^{(s_1-s_4-s_6)}E_3^{(s_2-s_5)}
E_3^{(s_3)}\big)  1_{(n-a,n-a,a,a)} \nonumber \\
&\scs  \refequal{\eqref{eq.com4}}
\scs \He(a-s_1)\He(a-s_2) \He(a-s_3) \He(a+s_1+s_2-s_4) \He(a+s_2+s_3-s_5)
\He(a+s_1+s_2+s_3-s_4-s_5-s_6)\times
\nonumber \\
& \scs
 \xsum{s_4,s_5,s_6}{}
 {s_2+s_1 \brack s_4}
 {s_2+s_3 \brack s_5}
 {s_1+2s_2+s_3-s_4-s_5 \brack s_6} {s_1+s_2-s_4 \brack s_1}
 {s_1+s_2+s_3-s_4-s_5-s_6 \brack s_1+s_2-s_4}
 \nonumber \\
&\scs  \qquad \quad \times
 {s_2+s_3-s_5 \brack s_3}
 {s_1+s_2+s_3-s_4-s_5-s_6 \brack s_2+s_3-s_5}
 F_3^{(s_1+s_2+s_3-s_4-s_5-s_6)}
 E_3^{(s_1+s_2+s_3-s_4-s_5-s_6)} 1_{(n-a,n-a,a,a)}. \nonumber
\end{align}
\endgroup
Then to complete the computation of $\sY_\cE(a)1_\vartheta$ from
equation~\eqref{eq.tref}, set  $\tau=s_1+s_2+s_3-s_4-s_5-s_6$ for simplicity
and use the above computation to simplify each term in the sum from
\eqref{eq.tref}
\begingroup
\allowdisplaybreaks
\begin{align}
& \scs 1_{(n,n,0,0)} E_2^{(a)}E_1^{(a)}  E_3^{(a)}E_2^{(a)}
  F_3^{(s_1+s_2+s_3-s_4-s_5-s_6)}
 E_3^{(s_1+s_2+s_3-s_4-s_5-s_6)}
  F_2^{(a)}F_3^{(a)}  F_1^{(a)}F_2^{(a)}1_{(n,n,0,0)}
  \nonumber \\
& \scs = E_2^{(a)}E_1^{(a)}E_3^{(a)}
\big(E_2^{(a)}  F_3^{(\tau)} \big)
 \big(E_3^{(\tau)}
  F_2^{(a)}\big)F_3^{(a)}F_1^{(a)}F_2^{(a)}1_{(n,n,0,0)}
  \nonumber \\
& \refequal{\eqref{eq.com4}}
\scs  E_2^{(a)}E_1^{(a)}E_3^{(a)}
\big( F_3^{(\tau)}E_2^{(a)} \big)
 \big(F_2^{(a)}E_3^{(\tau)}
\big)F_3^{(a)}F_1^{(a)}F_2^{(a)}1_{(n,n,0,0)}
  \nonumber \\
  & \scs \refequal{\eqref{eq.com0}}
E_2^{(a)}E_1^{(a)}\big(E_3^{(a)}
F_3^{(\tau)}1_{(n-a,n,\tau,a-\tau)} \big)E_2^{(a)}
F_2^{(a)}\big(E_3^{(\tau)}
F_3^{(a)} 1_{(n-a,n,a,0)}\big)F_1^{(a)}F_2^{(a)}
  \nonumber \\
& \refequal{\eqref{eq.com1}} \scs
\xsum{p_1,p_2}{} \scs {\tau \brack p_2}{\tau \brack p_1}
 E_2^{(a)}E_1^{(a)}\big(
F_3^{(\tau-p_2)}E_3^{(a-p_2)}1_{(n-a,n,\tau,a-\tau)} \big)E_2^{(a)}
F_2^{(a)}\big(
F_3^{(a-p_1)}E_3^{(\tau-p_1)} 1_{(n-a,n,a,0)}\big)F_1^{(a)}F_2^{(a)}
  \nonumber \\
& \scs =
\xsum{p_1,p_2}{} \scs {\tau \brack p_2}{\tau \brack p_1}
 1_{\vartheta} \big(E_2^{(a)}E_1^{(a)} F_3^{(\tau-p_2)}\big)
E_3^{(a-p_2)} E_2^{(a)}
F_2^{(a)} F_3^{(a-p_1)} \big(E_3^{(\tau-p_1)} F_1^{(a)}F_2^{(a)}\big)
1_{\vartheta}
  \nonumber \\
& \scs \refequal{\eqref{eq.com4}}
\xsum{p_1,p_2}{} \scs {\tau \brack p_2}{\tau \brack p_1}
 1_{\vartheta} \big(F_3^{(\tau-p_2)} E_2^{(a)}E_1^{(a)}\big)
E_3^{(a-p_2)} E_2^{(a)}
F_2^{(a)} F_3^{(a-p_1)} \big( F_1^{(a)}F_2^{(a)}E_3^{(\tau-p_1)}\big)
1_{\vartheta}
 \nonumber \\
& \scs =
\xsum{p_1,p_2}{} \scs {\tau \brack p_2}{\tau \brack p_1}
 \big(1_{\vartheta}F_3^{(\tau-p_2)}\big) E_2^{(a)}E_1^{(a)}
E_3^{(a-p_2)}  E_2^{(a)}
F_2^{(a)} F_3^{(a-p_1)} F_1^{(a)}F_2^{(a)}\big(E_3^{(\tau-p_1)}
1_{\vartheta}\big)
\nonumber \\
& \scs \refequal{\eqref{eq.Ei}}
\xsum{p_1,p_2}{} \scs {\tau \brack p_2}{\tau \brack p_1}
 \big(1_{\vartheta} \delta_{\tau,p_2}\big) E_2^{(a)}E_1^{(a)} E_3^{(a-p_2)}
E_2^{(a)}
F_2^{(a)} F_3^{(a-p_1)} F_1^{(a)}F_2^{(a)}
\big(\delta_{\tau,p_1} 1_{\vartheta}\big)
\nonumber \\
&\scs \refequal{\eqref{eq.com0}}
   E_2^{(a)}E_1^{(a)} E_3^{(a-\tau)}
  \big(E_2^{(a)}F_2^{(a)} 1_{\maltese(n-a,n,\tau,a-\tau)}  \big)
F_3^{(a-\tau)} F_1^{(a)}F_2^{(a)}  1_{\vartheta}
\nonumber \\
& \scs  \refequal{\eqref{eq.com1}}
\xsum{s_7}{} {n-\tau \brack s_7}
   E_2^{(a)}E_1^{(a)} E_3^{(a-\tau)}
  \big(F_2^{(a-s_7)} E_2^{(a-s_7)}1_{(n-a,n,\tau,a-\tau)}  \big)
F_3^{(a-\tau)} F_1^{(a)}F_2^{(a)}  1_{\vartheta}
  \nonumber \\
&\scs =
\xsum{s_7}{} {n-\tau \brack s_7}
   E_2^{(a)}\big(E_1^{(a)} E_3^{(a-\tau)}
F_2^{(a-s_7)}\big) \big(E_2^{(a-s_7)}F_3^{(a-\tau)} F_1^{(a)}\big)F_2^{(a)}
1_{\vartheta}
\nonumber \\
& \scs \refequal{\eqref{eq.com4}}
\xsum{s_7}{} {n-\tau \brack s_7}
   E_2^{(a)}\big(F_2^{(a-s_7)} E_1^{(a)} E_3^{(a-\tau)}
\big) \big( F_3^{(a-\tau)} F_1^{(a)} E_2^{(a-s_7)}\big)F_2^{(a)}
1_{\vartheta}
\nonumber \\
&\scs \refequal{\eqref{eq.com0}}
\xsum{s_7}{} {n-\tau \brack s_7}
   \big(E_2^{(a)}F_2^{(a-s_7)}1_{\maltese(n,n-s_7,a-s_7,0)}\big)
E_1^{(a)} E_3^{(a-\tau)}
F_3^{(a-\tau)} F_1^{(a)} \big( E_2^{(a-s_7)}F_2^{(a)}  1_{\vartheta}\big)
\nonumber \\
&\scs  \refequal{\eqref{eq.com1}}
\xsum{s_7,v_1,v_2}{} {n-\tau \brack s_7}
{n+s_7 \brack v_1}{n-s_7 \brack v_2}
   \big(1_{\vartheta}F_2^{(a-s_7-v_2)} E_2^{(a-v_2)} \big)
E_1^{(a)} E_3^{(a-\tau)}
F_3^{(a-\tau)} F_1^{(a)} \big(F_2^{(a-v_1)}  E_2^{(a-s_7-v_1)}
1_{\vartheta}\big)
\nonumber \\
& \scs =
\xsum{s_7,v_1,v_2}{} {n-\tau \brack s_7}
{n+s_7 \brack v_1}{n-s_7 \brack v_2}
   \big(1_{\vartheta}F_2^{(a-s_7-v_2)}\big) E_2^{(a-v_2)}E_1^{(a)}
E_3^{(a-\tau)}
F_3^{(a-\tau)} F_1^{(a)}F_2^{(a-v_1)}  \big(E_2^{(a-s_7-v_1)}1_{\vartheta}\big)
\nonumber \\
& \scs \refequal{\eqref{eq.Ei}}
\xsum{s_7,v_1,v_2}{} {n-\tau \brack s_7}
{n+s_7 \brack v_1}{n-s_7 \brack v_2}
   \big(1_{\vartheta}\delta_{v_2,a-s_7} \big) E_2^{(a-v_2)}E_1^{(a)}
E_3^{(a-\tau)}
F_3^{(a-\tau)} F_1^{(a)}F_2^{(a-v_1)}  \big(\delta_{v_1,a-s_7}
1_{\vartheta}\big)
\nonumber \\
&\scs =
\xsum{s_7}{} {n-\tau \brack s_7}
{n+s_7 \brack a-s_7}{n-s_7 \brack a-s_7}
     E_2^{(s_7)}E_1^{(a)} E_3^{(a-\tau)}
F_3^{(a-\tau)} F_1^{(a)}F_2^{(s_7)}    1_{\vartheta}
\nonumber \\
&\scs \refequal{\eqref{eq.com0}}
\xsum{s_7}{} {n-\tau \brack s_7}
{n+s_7 \brack a-s_7}{n-s_7 \brack a-s_7}
     E_2^{(s_7)}E_1^{(a)}
     \big(E_3^{(a-\tau)}F_3^{(a-\tau)}1_{(n-a,n+a-s_7,s_7,0)} \big)
F_1^{(a)}F_2^{(s_7)} 1_{\vartheta}
\nonumber \\
& \scs \refequal{\eqref{eq.com1}}
\xsum{s_7,v_3}{} {n-\tau \brack s_7}
{n+s_7 \brack a-s_7}{n-s_7 \brack a-s_7}{s_7 \brack v_3}
     E_2^{(s_7)}E_1^{(a)}
     \big(F_3^{(a-\tau-v_3)}E_3^{(a-\tau-v_3)}1_{(n-a,n+a-s_7,s_7,0)} \big)
F_1^{(a)}F_2^{(s_7)} 1_{\vartheta}
\nonumber \\
& \scs \refequal{\eqref{eq.com4}}
\xsum{s_7,v_3}{} {n-\tau \brack s_7}
{n+s_7 \brack a-s_7}{n-s_7 \brack a-s_7}{s_7 \brack v_3}
   \big(1_{\vartheta} F_3^{(a-\tau-v_3)}\big) E_2^{(s_7)}E_1^{(a)}
   F_1^{(a)}F_2^{(s_7)} \big(  E_3^{(a-\tau-v_3)}  1_{\vartheta} \big)
\nonumber \\
&\scs \refequal{\eqref{eq.Ei}}
\xsum{s_7,v_3}{} {n-\tau \brack s_7}
{n+s_7 \brack a-s_7}{n-s_7 \brack a-s_7}{s_7 \brack v_3}
   \big(1_{\vartheta} \delta_{v_3,a-\tau}\big) E_2^{(s_7)}E_1^{(a)}
   F_1^{(a)}F_2^{(s_7)} \big(  \delta_{v_3,a-\tau}   1_{\vartheta} \big)
   \nonumber \\
&\scs \refequal{\eqref{eq.com0}}
\xsum{s_7 }{} {n-\tau \brack s_7}
{n+s_7 \brack a-s_7}{n-s_7 \brack a-s_7}{s_7 \brack a-\tau}
     E_2^{(s_7)} \big(E_1^{(a)}
   F_1^{(a)} 1_{(n,n-s_7,s_7,0)}\big)F_2^{(s_7)}
\nonumber \\
&\scs \refequal{\eqref{eq.com1}}
\xsum{s_7,v_4}{} {n-\tau \brack s_7}
{n+s_7 \brack a-s_7}{n-s_7 \brack a-s_7}{s_7 \brack a-\tau} {s_7 \brack v_4}
     E_2^{(s_7)} \big(
   F_1^{(a-v_4)} E_1^{(a-v_4)}1_{(n,n-s_7,s_7,0)}\big)F_2^{(s_7)}
1_{\vartheta}
\nonumber \\
&\scs \refequal{\eqref{eq.com4}}
\xsum{s_7,v_4}{} {n-\tau \brack s_7}
{n+s_7 \brack a-s_7}{n-s_7 \brack a-s_7}{s_7 \brack a-\tau} {s_7 \brack v_4}
   \big( 1_{\vartheta} F_1^{(a-v_4)} \big)E_2^{(s_7)}
F_2^{(s_7)} \big(E_1^{(a-v_4)}   1_{\vartheta}\big)
\nonumber \\
&\scs \refequal{\eqref{eq.Ei}}
\xsum{s_7}{} {n-\tau \brack s_7}
{n+s_7 \brack a-s_7}{n-s_7 \brack a-s_7}{s_7 \brack a-\tau}
\maltese{s_7 \brack a} 1_{\vartheta} E_2^{(s_7)}
F_2^{(s_7)} 1_{\vartheta}
\end{align}
\endgroup
Tracing through this computation we have placed the symbol $\maltese$ to
indicate places where we must introduce Heaviside functions, so that the
end result should be multiplied by
\[
 \He(\tau)\He(a-\tau)\He(a-s_7) \He(s_7-a)
\]
from which, we deduce that $s_7=a$ from the quantum binomials appearing in
the summation.  The last Heaviside function of $\He(s_7-a)$ arises from
the definition of quantum binomial coefficients \eqref{eq.qbinom}. Thus,
the sum simplifies to
\begin{align}
& \scs   \He(\tau)\He(a-\tau) {n-\tau \brack a}
{n+a \brack 0}{n-a \brack 0}
{a \brack a-\tau}{a \brack a}
1_{(n,n,0,0)} E_2^{(a)}
F_2^{(a)} 1_{(n,n,0,0)} \nonumber \\
&= \scs   \He(\tau)\He(a-\tau) {n-\tau \brack a}
 {a \brack a-\tau}
1_{(n,n,0,0)} E_2^{(a)}
F_2^{(a)} 1_{(n,n,0,0)}  \nonumber \\
&\refequal{\eqref{eq.com1}}
\scs  \He(\tau)\He(a-\tau)\xsum{v_6}{}
 {n-\tau \brack a}
 {a \brack a-\tau}
 {n \brack v_6}
1_{(n,n,0,0)}
F_2^{(a-v_6)}E_2^{(a-v_6)} 1_{(n,n,0,0)}
\nonumber \\
&\refequal{\eqref{eq.Ei}}
\scs  \He(\tau)\He(a-\tau)
 {n-\tau \brack a}
 {a \brack a-\tau}
 {n \brack a}
1_{\vartheta}
\end{align}
Putting it altogether, the $a$-colored trefoil evaluates to
equation~\eqref{eq.trefoil}.

\section{Proof of parts (b) and  (c) of Proposition \ref{r.WL}}
\label{app.c}

For a compact oriented surface (possibly with boundary) $\Sigma$ let
$\calS(\Sigma)$ be the  \HOMFLY\
skein algebra of $\Sigma$, as defined in \cite{AM,Mo2}. Recall that as a
$\BQ(x,q)$-module,
$\calS(\Sigma)$ is generated by oriented links diagrams on $\Sigma$ modulo
the regular isotopy, the two relations
\[
\hackcenter{
\begin{tikzpicture}[scale =0.8]
\draw[very thick, directed=1] (-.5,0) -- (.5,1.5);
\draw[very thick ] (.5,0) -- (.1,.6);
\draw[very thick, directed=1 ] (-.1,.9) -- (-.5,1.5);
\end{tikzpicture}}
\quad - \quad
\hackcenter{
\begin{tikzpicture} [scale =0.8]
\draw[very thick, directed=1] (.5,0) -- (-.5,1.5);
\draw[very thick ] (-.5,0) -- (-.1,.6);
\draw[very thick, directed=1 ] (.1,.9) -- (.5,1.5);
\end{tikzpicture} }
\;\; = \;\;
(q-q^{-1}) \;\;
\hackcenter{
\begin{tikzpicture} [scale =0.8]
\draw[very thick, directed=1] (.5,0) .. controls ++(-.2,.4) and ++(-.2,-.4) .. (.5,1.5);
\draw[very thick,directed=1 ] (-.5,0)  .. controls ++(.2,.4) and ++(.2,-.4) .. (-.5,1.5);
\end{tikzpicture}}
\]
\[
\hackcenter{
\begin{tikzpicture} [scale =0.8]
\draw[very thick ] (0,-.25) -- (0,.5)
    .. controls ++(0,.4) and ++(0,.5) .. (.75,.5)
    .. controls ++(0,-.4) and ++(0,-.15) .. (.25,.5);
\draw[very thick, directed=1 ] (.1,.95) ..controls  ++(-.1,.3) and ++(0,-.3).. (0,1.75);
\end{tikzpicture}}
\;\; = \;\; x\;\;
\hackcenter{
\begin{tikzpicture} [scale =0.8]
\draw[very thick,directed=1 ] (0,-.25) -- (0,1.75);
\end{tikzpicture}}
\]
and the relation that a disjoint trivial knot can be removed from a diagram at the expense of multiplication with $\frac{x- x^{-1}}{q-q^{-1}}$. The product $L_1L_2$ of two links diagrams is the obtained by placing $L_1$ atop $L_2$.
When $\Sigma$ is a disk, $\calS(\Sigma) \cong \BQ(q,x)$ via a map $L \to \la L\ra$, where $\la L \ra$ is a framed version of the \HOMFLY\ polynomial.

The \HOMFLY\ skein algebra of the annulus contains the subalgebra $\calC^+$ generated by the closure of all braids. It is known that $\calC^+$ is isomorphic to the algebra of symmetric functions (with ground ring $\BQ(q,x)$). Under this isomorphism, the Schur function $s_\lambda$ of a partition $\lambda$ corresponds to a certain skein element $Q_\lambda$ which will be recalled later. The relation with the colored \HOMFLY\ polynomial is as follows. For an oriented link diagram $L$ on the disk with
 $r$ ordered components, and for partitions
$\lambda_i$ for $i=1,\dots,r$ we have:
\be
\label{eq.WLa}
W_L(\lambda_1,\dots,\lambda_r)=
\la L \ast (Q_{\lambda_1}, \dots, Q_{\lambda_r}) \ra
\ee
Here, $L \ast (Q_{\lambda_1}, \dots, Q_{\lambda_r})$ is the $\BQ(q,x)$-linear
combination of  link diagrams on the disk obtained by replacing the $i$th component
of $L$ by $Q_{\lambda_i}$. The above equality implies part (b) of Proposition \ref{r.WL}.


Let $\sigma: \BQ(x,q) \to \BQ(x,q)$ denote the $\BQ$-algebra automorphism given by
$\sigma(x)=x$, $\sigma(q)=-q^{-1}$.  One can easily check that $\sigma$ extends to a $\BQ$-linear automorphism of $\calS(\Sigma)$ for any $\Sigma$ by setting $\sigma(L):= L$ for any link diagram $L$ on $\Sigma$.
It is easy to see that $y$ is an element of the \HOMFLY\ skein algebra of the disk, then
 \be
\label{eq.sigmaL}
\sigma(\la y \ra)=\la \sigma (y)  \ra.
\ee

\begin{lemma} For any partition $\lambda$ one has
\be
\label{eq.Qlambda2}
\sigma(Q_{\lambda}) = Q_{\lambda^\dagger}\,.
\ee
\end{lemma}
\begin{proof}  In~\cite{AM}, Morton-Aiston gave a geometric description of $Q_\lambda$
in terms of closures of braids. Let us recall this formula for partitions
with one row $h_a=(a)$ or one column $e_a=(1^a)$ from~\cite[p.11]{AM}:
\be
\label{eq.Qaa}
Q_{(a)} = \frac{1}{\alpha_{(a)}}
\sum_{\pi \in \Sym_a} q^{l(\pi)} \widehat{\omega}_\pi, \qquad
Q_{(1^a)} = \frac{1}{\alpha_{(1^a)}}
\sum_{\pi \in \Sym_a} (-q^{-1})^{l(\pi)} \widehat{\omega}_\pi \,.
\ee
Here, for a permutation $\pi$ of $\Sym_n$, $\omega_\pi$ denotes the
positive braid corresponding to $\pi$, and $\widehat{\omega}_\pi \in \calC$ denotes the closure
of $\omega_\pi$. Moreover, $\alpha_\lambda$ is given by~\cite[p.14]{AM}
\be
\label{eq.alphalambda}
\alpha_\lambda = \prod_{(i,j) \in \lambda} q^{j-i} [\text{hook}(ij)]
\ee
where $\text{hook}(ij)$ is the hook-length of the cell $(i,j)$ of the
partition $\lambda$.

From  Equations~\eqref{eq.Qaa}
and~\eqref{eq.alphalambda} one can readily check that  $\sigma(Q_{(a)})=Q_{(1^a)}$, proving the lemma for the case $\lambda = h_a=(a)$. The case of general $\lambda$ can be proved similarly, using explicit formulas of $Q_\lambda$ as described in \cite{AM}. Alternatively, one can reduce the general case to the case of one row as follows. The two Jacobi-Trudy formulas
$$
s_{\lambda}=\det\left( (h_{\lambda_i+j-i})_{i,j=1}^\ell\right),
\qquad
s_{\lambda^\dagger}=\det\left( (e_{\lambda_i+j-i})_{i,j=1}^\ell\right),
$$
together with the case $\lambda=h_a$ implies the lemma for general partitions.
\end{proof}

Suppose $L$ is an  oriented link diagram $L$ on the disk with
 $r$ ordered components, and
$\lambda_i$ for $i=1,\dots,r$ are partitions. We have
\begin{align*}
\sigma (W_L(\lambda_1, \dots, \lambda_r))&= \sigma \left ( \la L * (Q_{\lambda_1}, \dots, Q_{\lambda_r})  \ra  \right )  \qquad \text{by  \eqref{eq.WLa}} \\
&= \la \sigma (L * (Q_{\lambda_1}, \dots, Q_{\lambda_r}) )  \ra   \qquad \text{by  \eqref{eq.sigmaL}} \\
&= \la L * (\sigma (Q_{\lambda_1}), \dots, \sigma (Q_{\lambda_r}) )  \ra\\
&=  \la L * (Q_{\lambda_1^\dagger}, \dots, Q_{\lambda_r^\dagger}) )  \qquad \text{by  \eqref{eq.Qlambda2}} \\
&= W_L(\lambda_1^\dagger, \dots, \lambda_r^\dagger).
\end{align*}
This  concludes the proof of part (c).
\qed

\section{The recursion for the colored \HOMFLY\ of the trefoil}
\label{app.d}

Let $\lambda\in \cP_{n-1}$ be a partition of length $\le n-1$. We also use
$\lambda$ to denote the corresponding $\Uqsn$-module. For every positive
integer $k$, the theory of ribbon categories gives a representation
$J: \fB_k \to \Aut(\lambda^{\ot k})$, where $\fB_k$ is the braid group on $k$
strands and $\Aut(\lambda^{\ot k})$ is the group of $\Uqsn$-linear
automorphisms of $\lambda^{\ot k}$.

Suppose $\beta\in B_m$ is a braid on $m$ strands, and $L=\cl(\beta)$ is the
oriented framed link obtained by closing $\beta$ in the standard way, with
blackboard framing. Then
\be
\label{eq.l3}
J_L(\lambda,\lambda,\dots) = \tr_q^{\lambda^{\ot m}}(J(\beta)),
\ee
where for a $\Uqsn$-linear map $f: V \to V$,
$$
\tr_q^V(f) = \tr(f \bg, V)\,.
$$
Here the right hand side is the usual trace of $f \bg$ acting on $V$, and
$\bg\in \Uqsn$ is the so-called {\em charm element} whose exact formula is
not needed here. In particular, for a finite-dimensional weight a
$\Uqsn$-module $V$, the quantum dimension $\dim_q(V):= J_U(V)$
(where $U$ is the unknot) is
$$
\dim_q(V) := \tr_q^{V}(\id) =\tr(\fg,V)\,.
$$

Let $\sigma$ be the standard generator of $\fB_2$
(see $X_{a,b}$ of Figure~\ref{fig:braiding}). Then $J(\sigma)$ is defined by
the universal $R$-matrix and the action of $J(\sigma)$ on $h_m ^{\ot 2}$ can
be calculated as follows. The decomposition of $h_m ^{\ot 2}$ into irreducible
$\Uqsn$-modules has the form
$$
h_m ^{\ot 2} = \bigoplus_{k=0}^m  \mu_{m,k}\,,
$$
where $\mu_{m,k}$ is the partition $(2m-k,k)$. Since $J(\sigma)$ is
$\Uqsn$-linear, Schur lemma shows that  there are scalars $c_{m,k}\in \Qqn$
such that on $h_m ^{\ot 2}$,
\be
\label{eq.l1}
J(\sigma)|_{h_m ^{\ot 2}} =\bigoplus_{k=0}^m c_{m,k} \id_{\mu_{m,k}}.
\ee

One of the axioms of the ribbon structure of $\Uqsn$  is that
\be
\label{eq.l}
J(\sigma^2)|_{V\ot W}= (\br_V ^{-1} \ot \br_{W}^{-1}) \br_{V\ot W},
\ee
where $\br$ is the ribbon element, which belongs to the center of a certain
completion of $\Uqsn$ and acts on any finite-dimensional weight
$\Uqsn$-module, see \cite{Tu2,Ohtsuki}. Geometrically, $\br=J_{T}$, where
$T$ is the trivial 1-1 tangle with framing 1, and its action on $\lambda$ is
known (see e.g. \cite[Equ 1.7]{Le:Integral}):
\be
\label{eq.twist}
\br|_{V_\lambda} =  \br(\lambda) \id_\lambda,\quad \text{where}\
\br(\lambda)= q ^{ \la \lambda, \lambda + 2 \rho \ra}.
\ee
Here $\la \cdot, \cdot \ra$ is the inner product on the weight space of
$\Uqsn$ normalized such that each root has square length 2, and $2\rho$ is
the sum of all positive roots.

Using \eqref{eq.l} in the square of \eqref{eq.l1}, we get
$$
(c_{m,k})^2 = \br(\mu_{m,k})\, \br(h_m)^{-2}\,.
$$
Taking the square root and using \eqref{eq.twist}), one gets the value of
$(c_{m,k})$, up to sign $\pm1$. The sign can be determined by noting that when
$q=1$, $J(\sigma)$ is the permutation, $J(\sigma)(x_1\ot x_2) = x_2 \ot x_1$.
Eventually, we get
\be
\label{eq.R1}
c_{m,k}= (-1)^k q^{-m^2/n }  q^{m^2 -2mk + k^2 -k}.
\ee
Suppose $T_{s}$ is the link closure of $\sigma^s$, which is a torus link of
type $(2,s)$. By~\eqref{eq.l3} and the decomposition \eqref{eq.l1},
\begin{align}
\tJ_{T_s}(h_m) &=  q^{s m^2/n} J_{T_s}(h_m)=q^{s m^2/n}
\sum_{k=0}^m (c_{m,k})^s \dim_q(\mu_{m,k})
\nonumber\\
&= \sum_{k=0}^m  (-1)^{sk} q^{s(m^2 -2mk + k^2 -k)} \dim_q(\mu_{m,k})
\nonumber\\
&= \sum_{k=0}^m(-1)^{sk} q^{s( m^2-2mk+k^2 -k )}
\qbinom{x;k-2}{k}\qbinom{x;2m-k-1}{2m-k} \frac{[2m-2k+1]}{[2m-k+1]}
\label{eq.tor},
\end{align}
where $x=q^n$. In the last equality we use the well-known formula for the
quantum dimension,  see e.g. \cite[Equ. (11)]{Mo1}, which was first
established by Reshetikhin.
The right hand side of \eqref{eq.tor} gives a formula for $W_{T_r}(h_m)$.
When $s=3$, we get another formula of $W_{T_3}$ for the trefoil, which is
simpler than the one given in Section \ref{sub.trefoil}, since it is a
one-dimensional sum.

For odd $s$, let $\oT_s$ be the torus knot $T_s$ with 0 framing. Then,
adjusting the framing, from \eqref{eq.tor} we get
\be
\label{eq.tor2}
W_{\oT_s}(h_m)= \sum_{k=0}^m(-1)^{k} q^{s(-m-2mk+k^2 -k )}
\qbinom{x;k-2}{k}\qbinom{x;2m-k-1}{2m-k} \frac{[2m-2k+1]}{[2m-k+1]}.
\ee
Using  the Zeilberger algorithm \cite{PWZ}, we get the recurrence
relation for $W_{\oT_3}(h_m)$ as described in Section~\ref{sub.rec31}.

\bibliographystyle{hamsalpha}
\bibliography{biblio}

\newcommand{\etalchar}[1]{$^{#1}$}
\providecommand{\bysame}{\leavevmode\hbox to3em{\hrulefill}\thinspace}
\providecommand{\href}[2]{#2}
\providecommand{\eprint}{\begingroup \urlstyle{rm}\Url}
\begin{thebibliography}{IMMM12b}

\bibitem[AENV14]{AVNg}
Mina Aganagic, Tobias Ekholm, Lenhard Ng, and Cumrun Vafa, \emph{Topological
  strings, {D}-model, and knot contact homology}, Adv. Theor. Math. Phys.
  \textbf{18} (2014), no.~4, 827--956.

\bibitem[AM98]{AM}
A.~K. Aiston and H.~R. Morton, \emph{Idempotents of {H}ecke algebras of type
  {$A$}}, J. Knot Theory Ramifications \textbf{7} (1998), no.~4, 463--487.

\bibitem[AV]{AV}
Mina Aganagic and Cumrun Vafa, \emph{Large {$N$} {D}uality, {M}irror
  {S}ymmetry, and a {Q}-deformed {A}-polynomial for {K}nots},
  \eprint{arXiv:1204.4709}, Preprint 2004.

\bibitem[BCL14]{BCL}
Anna Beliakova, Qi~Chen, and Thang T.~Q. Le, \emph{On the integrality of the
  {W}itten-{R}eshetikhin-{T}uraev 3-manifold invariants}, Quantum Topol.
  \textbf{5} (2014), no.~1, 99--141.

\bibitem[Che]{Cherednik}
Ivan Cherednik, \emph{{DAHA}-{J}ones polynomials of torus knots},
  \eprint{arXiv:1406.3959}, Preprint 2014.

\bibitem[CKM14]{CKM}
Sabin Cautis, Joel Kamnitzer, and Scott Morrison, \emph{Webs and quantum skew
  {H}owe duality}, Math. Ann. \textbf{360} (2014), no.~1-2, 351--390.

\bibitem[CP94]{CP}
Vyjayanthi Chari and Andrew Pressley, \emph{A guide to quantum groups},
  Cambridge University Press, Cambridge, 1994.

\bibitem[CR08]{CR}
Joseph Chuang and Rapha{\"e}l Rouquier, \emph{Derived equivalences for
  symmetric groups and {$\mathfrak{sl}_2$}-categorification}, Ann. of Math. (2)
  \textbf{167} (2008), no.~1, 245--298.

\bibitem[DG10]{DG:refined}
Tudor Dimofte and Sergei Gukov, \emph{Refined, motivic, and quantum}, Lett.
  Math. Phys. \textbf{91} (2010), no.~1, 1--27.

\bibitem[DGR06]{Dunfield-Gukov}
Nathan~M. Dunfield, Sergei Gukov, and Jacob Rasmussen, \emph{The
  superpolynomial for knot homologies}, Experiment. Math. \textbf{15} (2006),
  no.~2, 129--159.

\bibitem[EGNO15]{Etingof-tensor}
Pavel Etingof, Shlomo Gelaki, Dmitri Nikshych, and Victor Ostrik, \emph{Tensor
  categories}, Mathematical Surveys and Monographs, vol. 205, American
  Mathematical Society, Providence, RI, 2015.

\bibitem[FGS13]{Fuji1}
Hiroyuki Fuji, Sergei Gukov, and Piotr Su{\l}kowski,
  \emph{Super-{$A$}-polynomial for knots and {BPS} states}, Nuclear Phys. B
  \textbf{867} (2013), no.~2, 506--546.

\bibitem[FGSA12]{Fuji2}
Hiroyuki Fuji, Sergei Gukov, Piotr Su{\l}kowski, and Hidetoshi Awata,
  \emph{Volume conjecture: refined and categorified}, Adv. Theor. Math. Phys.
  \textbf{16} (2012), no.~6, 1669--1777.

\bibitem[FH91]{FH}
William Fulton and Joe Harris, \emph{Representation theory}, Graduate Texts in
  Mathematics, vol. 129, Springer-Verlag, New York, 1991, A first course,
  Readings in Mathematics.

\bibitem[FYH{\etalchar{+}}85]{HOMFLY}
Petrer Freyd, David Yetter, Jim Hoste, Raymond Lickorish, Ken Millett, and
  Adrian Ocneanu, \emph{A new polynomial invariant of knots and links}, Bull.
  Amer. Math. Soc. (N.S.) \textbf{12} (1985), no.~2, 239--246.

\bibitem[Gara]{Ga:homfly}
Stavros Garoufalidis, \emph{The colored {HOMFLY} polynomial is
  {$q$}-holonomic}, \eprint{arXiv:1211.6388}, Preprint 2012.

\bibitem[Garb]{Ga3}
Stavros Garoufalidis, \emph{Quantum knot invariants}, Mathematische
  Arbeitstagung 2012, \eprint{arXiv:1201.3314}.

\bibitem[Gar04]{Ga:AJ}
Stavros Garoufalidis, \emph{On the characteristic and deformation varieties of
  a knot}, Proceedings of the {C}asson {F}est, Geom. Topol. Monogr., vol.~7,
  Geom. Topol. Publ., Coventry, 2004, pp.~291--309 (electronic).

\bibitem[GL]{GL:survey}
Stavros Garoufalidis and Thang T.~Q. L{\^e}, \emph{A survey of {$q$}-holonomic
  functions}, Preprint 2015.

\bibitem[GL05]{GL}
\bysame, \emph{The colored {J}ones function is {$q$}-holonomic}, Geom. Topol.
  \textbf{9} (2005), 1253--1293 (electronic).

\bibitem[GS12]{Gukov-Stosic}
Sergei Gukov and Marko Sto{\v{s}}i{\'c}, \emph{Homological algebra of knots and
  {BPS} states}, Proceedings of the {F}reedman {F}est, Geom. Topol. Monogr.,
  vol.~18, Geom. Topol. Publ., Coventry, 2012, pp.~309--367.

\bibitem[GS14]{Gukov:lectures}
Sergei Gukov and Ingmar Saberi, \emph{Lectures on knot homology and quantum
  curves}, Topology and field theories, Contemp. Math., vol. 613, Amer. Math.
  Soc., Providence, RI, 2014, pp.~41--78.

\bibitem[GSV05a]{GSV}
Sergei Gukov, Albert Schwarz, and Cumrun Vafa, \emph{Khovanov-{R}ozansky
  homology and topological strings}, Lett. Math. Phys. \textbf{74} (2005),
  no.~1, 53--74.

\bibitem[GSV05b]{Gukov-Vafa}
\bysame, \emph{Khovanov-{R}ozansky homology and topological strings}, Lett.
  Math. Phys. \textbf{74} (2005), no.~1, 53--74.

\bibitem[GvdV14]{GV}
Stavros Garoufalidis and Roland van~der Veen, \emph{A generating series for
  {M}urakami-{O}htsuki-{Y}amada graph evaluations}, Acta Math. Vietnam.
  \textbf{39} (2014), no.~4, 529--539.

\bibitem[GZ]{GZ}
Stavros Garoufalidis and Don Zagier, \emph{Quantum modularity of the kashaev
  invariant}, Preprint 2013.

\bibitem[IMMM12a]{Ito2}
Hiroshi Itoyama, Andrei Mironov, Alexei Morozov, and Andrey Morozov,
  \emph{Character expansion for {HOMFLY} polynomials {III}. {A}ll 3-strand
  braids in the first symmetric representation}, Internat. J. Modern Phys. A
  \textbf{27} (2012), no.~19, 1250099, 85.

\bibitem[IMMM12b]{Ito1}
\bysame, \emph{H{OMFLY} and superpolynomials for figure eight knot in all
  symmetric and antisymmetric representations}, J. High Energy Phys. (2012),
  no.~7, 131, front matter+20.

\bibitem[Jan96]{Jantzen}
Jens~Carsten Jantzen, \emph{Lectures on quantum groups}, Graduate Studies in
  Mathematics, vol.~6, American Mathematical Society, Providence, RI, 1996.

\bibitem[Jon]{Jones2}
Vaughan Jones, \emph{Planar algebras, i}, \eprint{arXiv:9909027}, Preprint
  1999.

\bibitem[Jon87]{Jones}
\bysame, \emph{Hecke algebra representations of braid groups and link
  polynomials}, Ann. of Math. (2) \textbf{126} (1987), no.~2, 335--388.

\bibitem[Kaw]{Kawagoe}
Kenichi Kawagoe, \emph{On the formulae for the colored {HOMFLY} polynomials},
  \eprint{arXiv:1210.7574}, Preprint 2012.

\bibitem[KRT97]{Kassel}
Christian Kassel, Marc Rosso, and Vladimir Turaev, \emph{Quantum groups and
  knot invariants}, Panoramas et Synth\`eses [Panoramas and Syntheses], vol.~5,
  Soci\'et\'e Math\'ematique de France, Paris, 1997.

\bibitem[Kup96]{Kuperberg}
Greg Kuperberg, \emph{Spiders for rank {$2$} {L}ie algebras}, Comm. Math. Phys.
  \textbf{180} (1996), no.~1, 109--151.

\bibitem[L{\^e}00]{Le:Integral}
Thang~T.Q. L{\^e}, \emph{Integrality and symmetry of quantum link invariants},
  Duke Math. J. \textbf{102} (2000), no.~2, 273--306.

\bibitem[L{\^e}06]{Le:AJ14}
\bysame, \emph{The colored {J}ones polynomial and the {$A$}-polynomial of
  knots}, Adv. Math. \textbf{207} (2006), no.~2, 782--804.

\bibitem[LMV00]{LMOV}
Jos{\'e} Labastida, Marcos Mari{\~n}o, and Cumrun Vafa, \emph{Knots, links and
  branes at large {$N$}}, J. High Energy Phys. (2000), no.~11, Paper 7, 42.

\bibitem[Luk01]{Lucak}
Sascha~G. Lukac, \emph{Homfly skeins and the {H}opf link}, 2001, Thesis,
  University of Liverpool.

\bibitem[Lus10]{Lu}
George Lusztig, \emph{Introduction to quantum groups}, Modern Birkh\"auser
  Classics, Birkh\"auser/Springer, New York, 2010, Reprint of the 1994 edition.

\bibitem[Mac95]{Mc}
Ian Macdonald, \emph{Symmetric functions and {H}all polynomials}, second ed.,
  Oxford Mathematical Monographs, The Clarendon Press Oxford University Press,
  New York, 1995, With contributions by A. Zelevinsky, Oxford Science
  Publications.

\bibitem[ML03]{Mo1}
Hugh Morton and Sascha Lukac, \emph{The {H}omfly polynomial of the decorated
  {H}opf link}, J. Knot Theory Ramifications \textbf{12} (2003), no.~3,
  395--416.

\bibitem[MM08]{Mo2}
Hugh Morton and Pedro Manch{\'o}n, \emph{Geometrical relations and plethysms in
  the {H}omfly skein of the annulus}, J. Lond. Math. Soc. (2) \textbf{78}
  (2008), no.~2, 305--328.

\bibitem[MOY98]{MOY}
Hitoshi Murakami, Tomotada Ohtsuki, and Shuji Yamada, \emph{Homfly polynomial
  via an invariant of colored plane graphs}, Enseign. Math. (2) \textbf{44}
  (1998), no.~3-4, 325--360.

\bibitem[MW98]{MW}
Gregor Masbaum and Hans Wenzl, \emph{Integral modular categories and
  integrality of quantum invariants at roots of unity of prime order}, J. Reine
  Angew. Math. \textbf{505} (1998), 209--235.

\bibitem[Ng08]{Ng}
Lenhard Ng, \emph{Framed knot contact homology}, Duke Math. J. \textbf{141}
  (2008), no.~2, 365--406.

\bibitem[NRZS12]{NRS}
Satoshi Nawata, P.~Ramadevi, Zodinmawia, and Xinyu Sun,
  \emph{Super-{A}-polynomials for twist knots}, J. High Energy Phys. (2012),
  no.~11, 157, front matter + 38.

\bibitem[Oht02]{Ohtsuki}
Tomotada Ohtsuki, \emph{Quantum invariants}, Series on Knots and Everything,
  vol.~29, World Scientific Publishing Co., Inc., River Edge, NJ, 2002, A study
  of knots, 3-manifolds, and their sets.

\bibitem[PT87]{PT}
J{\'o}zef Przytycki and Pawe{\l} Traczyk, \emph{Conway algebras and skein
  equivalence of links}, Proc. Amer. Math. Soc. \textbf{100} (1987), no.~4,
  744--748.

\bibitem[PWZ96]{PWZ}
Marko Petkov{\v{s}}ek, Herbert~S. Wilf, and Doron Zeilberger, \emph{{$A=B$}}, A
  K Peters Ltd., Wellesley, MA, 1996, With a foreword by Donald E. Knuth, With
  a separately available computer disk.

\bibitem[QSa]{QS2}
Hoel Queffelec and Antonio Sartori, \emph{Homfly-pt and alexander polynomials
  from a doubled schur algebra}, \eprint{arXiv:1412.3824}, Preprint 2014.

\bibitem[QSb]{QS}
\bysame, \emph{Mixed quantum skew {H}owe duality and link invariants of type
  {A}}, \eprint{arXiv:1504.01225}, Preprint 2015.

\bibitem[RT90]{RT}
Nicolai Reshetikhin and Vladimir Turaev, \emph{Ribbon graphs and their
  invariants derived from quantum groups}, Comm. Math. Phys. \textbf{127}
  (1990), no.~1, 1--26.

\bibitem[Sab93]{Sabbah}
Claude Sabbah, \emph{Syst\`emes holonomes d'\'equations aux
  {$q$}-diff\'erences}, {$D$}-modules and microlocal geometry ({L}isbon, 1990),
  de Gruyter, Berlin, 1993, pp.~125--147.

\bibitem[Tur88]{Tu1}
Vladimir Turaev, \emph{The {Y}ang-{B}axter equation and invariants of links},
  Invent. Math. \textbf{92} (1988), no.~3, 527--553.

\bibitem[Tur94]{Tu2}
\bysame, \emph{Quantum invariants of knots and 3-manifolds}, de Gruyter Studies
  in Mathematics, vol.~18, Walter de Gruyter \& Co., Berlin, 1994.

\bibitem[Wed]{Wedrich}
Paul Wedrich, \emph{{$q$}-holonomic formulas for colored {HOMFLY} polynomials
  of 2-bridge links}, \eprint{arXiv:1410.3769}, Preprint 2014.

\bibitem[Wen90]{Wenzl}
Hans Wenzl, \emph{Representations of braid groups and the quantum
  {Y}ang-{B}axter equation}, Pacific J. Math. \textbf{145} (1990), no.~1,
  153--180.

\bibitem[WZ92]{WZ}
Herbert Wilf and Doron Zeilberger, \emph{An algorithmic proof theory for
  hypergeometric (ordinary and ``{$q$}'') multisum/integral identities},
  Invent. Math. \textbf{108} (1992), no.~3, 575--633.

\bibitem[Zei90]{Z90}
Doron Zeilberger, \emph{A holonomic systems approach to special functions
  identities}, J. Comput. Appl. Math. \textbf{32} (1990), no.~3, 321--368.

\end{thebibliography}
\end{document}